\documentclass{amsart}

\usepackage[shortalphabetic]{amsrefs}
\usepackage{amsmath}
\usepackage{amsthm}
\usepackage{amssymb}
\usepackage{multirow} 
\usepackage{enumerate}
\usepackage[all]{xy}
\usepackage{graphicx}
\usepackage{comment}
\usepackage{xcolor}
\usepackage{booktabs}

\numberwithin{figure}{section}

\numberwithin{equation}{section}
\newtheorem{thm}[equation]{Theorem}
\newtheorem*{thm*}{Theorem}
\newtheorem{lem}[equation]{Lemma}
\newtheorem{prop}[equation]{Proposition}
\newtheorem{cor}[equation]{Corollary}
\newtheorem{lemma}[equation]{Lemma}

\theoremstyle{definition}
\newtheorem{defn}[equation]{Definition}

\theoremstyle{remark}
\newtheorem{remark}[equation]{Remark}
\newtheorem{ex}[equation]{Example}

\DeclareMathOperator{\End}{End}
\DeclareMathOperator{\Hom}{Hom}
\DeclareMathOperator{\Aut}{Aut}
\DeclareMathOperator{\isom}{Isom}

\DeclareMathOperator{\GL}{GL}
\DeclareMathOperator{\PGL}{PGL}
\DeclareMathOperator{\diag}{diag}
\DeclareMathOperator{\disc}{disc}
\DeclareMathOperator{\Pf}{Pf}
\DeclareMathOperator{\GammaL}{\Gamma L}
\DeclareMathOperator{\Gal}{Gal}

\newcommand{\M}{\mathbb{M}}
\newcommand{\bmto}{\rightarrowtail}
\newcommand{\F}{\mathbb{F}_q}

\newcommand{\la}{\langle}
\newcommand{\ra}{\rangle}
\newcommand{\tr}{{\rm tr}}

\newcommand{\pseudo}{\Psi\hspace*{-1mm}\isom}
\newcommand{\V}{\F^{\,d}}

\renewcommand{\phi}{\varphi}

\renewcommand{\leq}{\leqslant}
\renewcommand{\geq}{\geqslant}
\renewcommand{\hom}{\Hom}

\title[Isomorphism of groups of genus 2]{A fast isomorphism test for groups of genus 2}
\author{Peter A. Brooksbank}
\address{
	Department of Mathematics\\
	Bucknell University\\
	Lewisburg, PA 17837\\
	USA
}
\email{pbrooksb@bucknell.edu}
\author{Joshua Maglione}
\address{
	Department of Mathematics\\
	Colorado~State~University\\
	Fort~Collins, CO 80523\\
	USA
}
\email{maglione@math.colostate.edu}
\author{James B. Wilson}
\address{
	Department~of~Mathematics\\
	Colorado State University\\
	Fort Collins, CO 80523\\
	USA
}
\email{James.Wilson@ColoState.Edu}

\begin{document}

\begin{abstract}
Motivated by the need for efficient isomorphism tests for
finite groups, 
we present a polynomial-time 
method for deciding isomorphism within a class of groups
that is well-suited to studying local properties of general finite groups.
We also report on the performance of an implementation of 
the algorithm in the computer
algebra system {\sc magma}.  
\end{abstract}

\maketitle


\section{Introduction}
\label{sec:intro}
This paper concerns the problem of testing whether 
two given finite groups are isomorphic. Work on the {\em group isomorphism problem}
has led to the development of many fundamental concepts in the modern theory of
groups -- Hall and Fitting subgroups, isoclinism, and coclass theory are some examples.
The problem itself has different aspects, ranging from
practical methods for use in the sciences \citelist{\cite{CH}\cite{ELGOB}}, 
to questions of computability~\citelist{\cite{Rabin}\cite{HL}\cite{BCGQ}\cite{LW}}, 
to the intimate but complex relationship  
it has with the {\em graph isomorphism problem}~\citelist{\cite{Miller}*{p. 132}\cite{HL}\cite{LV}*{Theorem~3.1}}.
The classification of finite simple groups, combined with natural recursive methods based 
on Sylow subgroups and the lower central series, gives a reduction to the case of $p$-groups
of exponent $p$-class $2$. Here, though, one hits a wall. 
The only general purpose techniques are variants of the 
{\em nilpotent quotient algorithm}~\citelist{\cite{OBrien}\cite{ELGOB}}, 
which in the worst cases requires $O(\exp(c_p \cdot d(G)^2))$ operations,  
where $d(G)\leq \log_p |G|$ is the size of a minimal generating set for $G$.\footnote{Most groups of order $p^n$ 
have $d(G)\approx 2n/3$~\cite{BNV:enum}*{pp. 26 \& 44}.}  
On the other hand, new techniques yield
isomorphism tests for some families of $p$-groups with unbounded
$d(G)$ that use just $O(\log^6 |G|)$ operations~\cite{LW}. 

Over a decade ago an idea emerged for a ``local-to-global" approach to isomorphism
testing of $p$-groups.  By examining many small, overlapping subgroups and quotients 
of the given groups, one aims to deduce constraints on isomorphisms between the groups themselves.  
The idea was discussed in greater detail at an Oberwolfach meeting in 2011, which in turn led to 
a collaboration of the first and third authors with E.A. O'Brien to build the infrastructure for such a 
test. 
The current work presents a nearly optimal isomorphism test for the family of groups ideal for use
as the ``local'' constituents of such a local-to-global isomorphism test. 

In order to present our main result in a concrete but sufficiently general setting, we assume
that groups are given, as in~\cite{KL:quo}, as quotients of permutation groups. 
Specifically,  a group $G$ is specified by sets 
$X$ and $R$ of permutations of $n$ letters, where $G=\la X\mid R\ra=\la X\ra/\la R^{\la X\ra}\ra$.
Thus, the length of an input is $\Theta((|X|+|R|)n)$, and an algorithm is {\em polynomial-time} if
it requires $O((|X|+|R|)n^c)$ operations, which can be as small as $O(\log^c |G|)$.
We prove the following result.

\begin{thm}
\label{thm:main1}
There is a deterministic, polynomial-time algorithm that, given groups 
$G$ and $H$, decides if $G$ is a $p$-group of exponent $p$-class $2$ and
commutator of order dividing $p^2$, and if so, decides if $G$ and $H$ are isoclinic.  
If $G$ has exponent $p$, any such isoclinism is an isomorphism.
\end{thm}

We shall eventually state and prove a stronger result (Theorem~\ref{thm:gp-alg}).   
This concerns
a broader family of groups that we call {\em genus} 2 (borrowing the term from Knebelman's 
work on  Lie algebras). 
Roughly speaking, these are direct products of exponent $p$-class 2 groups 
whose commutator structure may be encoded over some extension 
$\mathbb{F}_q$ of $\mathbb{Z}_p$, and whose commutator subgroup is isomorphic to 
$\mathbb{F}_q\times\mathbb{F}_q$. In particular,
there is no bound on the minimum number of generators
either of the genus 2 groups themselves, or of their commutator subgroups.  

Our methods also apply to a broader range of computational models
than the permutation group quotient model stated in Theorem~\ref{thm:main1}.
There are several standard models for computation with groups, such as
matrix groups and polycyclic presentations,
and our results apply (with minor qualifications) equally well to all of them. 
These models are discussed at greater length in Section~\ref{sec:models},
but it's important to stress two things -- first, that the order of a group is usually 
{\em exponential in its input length}, and secondly that the complexity of our
algorithms is polynomial {\em in the input length}, not in the order of the group.

\subsection{Computer implementation}
\label{sec:imp-intro}
As our principal objective in the local-to-global project is to produce 
practical algorithms for computing with $p$-groups,
we developed an implementation of the algorithm underlying Theorem~\ref{thm:main1}
in the computer algebra system {\sc magma}~\cite{magma}. Further details of the
implementation are given in Section~\ref{sec:imp}, but to illustrate the efficacy
of our methods, a sample of runtimes for 5-groups is given in
Figure~\ref{fig:plot}.  

\begin{figure}[!htbp]
\input{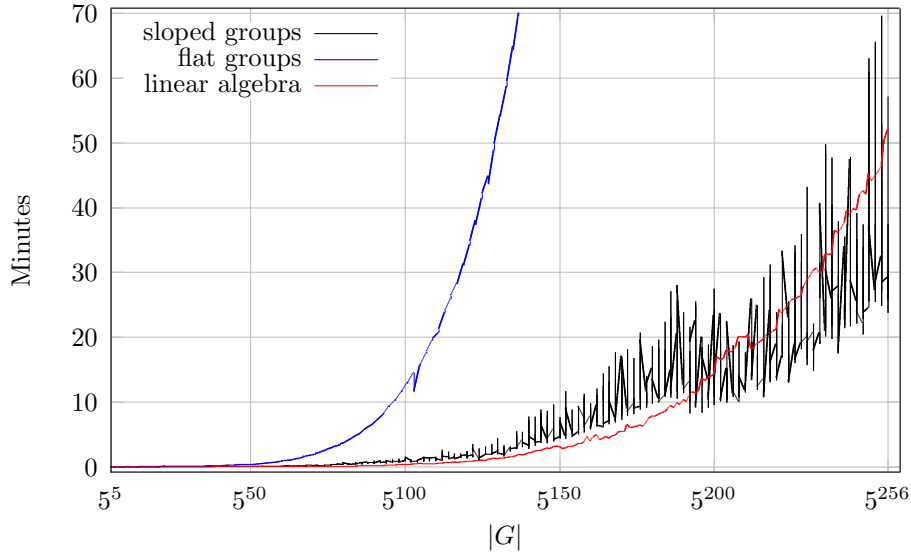}
\caption{Performance data for tests to confirm isomorphism
between groups of order $5^{d+2}$ for increasing $d$. The
``linear algebra" plot is the time needed by {\sc magma} to solve
linear $d^2\times d^2$ systems, shown here to track with our runtimes for
sloped genus 2 groups.}
\label{fig:plot}
\end{figure}

Curiosity led us to this demonstration -- we wanted to push
the orders of the input groups as high as possible while successfully
testing isomorphism within one hour.
We constructed 1260 random 5-groups of genus 2 having
orders ranging from $5^5$ to $5^{256}$, and generated for each
a random isomorphic copy. We then tasked our
implementation with finding an explicit isomorphism between
each pair of groups and plotted the completion time on the graph.
We intended to compare the performance of our implementation 
with that of default functions in {\sc magma}, but the 500GB
of memory on one of our machines was insufficient for the latter
to decide isomorphism even for groups of order $5^7$. Our algorithm
required less than 200 MB.

As the graphic in Figure~\ref{fig:plot} suggests, genus 2 groups come in two flavors --
``sloped" and ``flat". These terms will soon be defined precisely, but it suffices now to say
that in our experiment a group of order $5^d$ is always flat if $d$ is odd, 
and is generically sloped if $d$ is even. Although the performance evidently varies according
to type, both methods track with the cost of solving systems of  linear equations in approximately $d^2$ equations
and varaibles.
This is fairly clear for the sloped groups from Figure~\ref{fig:plot}; more 
refined graphics in Section~\ref{sec:imp} show that the same is true for
flat types. These data support the claim proved later in the paper that the asymptotic complexity of our algorithms
is $O(d^{2\omega})$, where $2\leq \omega<3$ is the exponent of matrix multiplication; cf. \cite{vzG}*{Chapter 12}.

We mention, finally, that we only aborted the experiment 
when we found that certain functions in {\sc magma} do not yet handle 
$p$-groups of order larger than $p^{256}$.  We view this as confirmation 
that our algorithms are adequate for practical use.

\subsection{Attacking the general isomorphism problem.}
The details of how the groups in Theorem~\ref{thm:main1} will be used within the 
local-to-global isomorphism test are the subject of a forthcoming article~\cite{BOW}, 
so we give just a brief summary of the properties that make them well-suited.
The difficulty with $p$-groups of class 2 is dealing with many 
commutator relations.  
If we work only with abelian quotients, we lose all commutator information.
Extraspecial quotients are the obvious class to look at next, 
but these groups have only two isomorphism
types (for any fixed order $p^{1+2m}$) and hence do not capture sufficient variability.  
Using groups with central
commutator isomorphic to $\mathbb{Z}_p\times \mathbb{Z}_p$ we can record 
commutator relations as points on a projective line. This, as we shall see,
admits surprising variability, but not so much as to make the problem intractable.

We expect to repeat local analysis many times within one global 
problem. We therefore need algorithms for the local versions to be extremely fast, and 
Figure~\ref{fig:plot} suggests they are. From a theoretical viewpoint,
the algorithm in Theorem~\ref{thm:main1} has 
complexity $O(d^{2\omega})$, where
$d=\log_p |G|-2$. This  
represents an exponential speed-up over existing $O(\exp(c d^2))$-time algorithms.
Thus, notwithstanding the constrained class of groups it handles,
Theorem~\ref{thm:main1} dramatically exceeds expectations.

\subsection{Classification problems}
Interest in the class of groups of Theorem~\ref{thm:main1} extends beyond isomorphism testing.
Perhaps most notable is that its
classification problem lies, in a technical sense, on the cusp of tractability. 
More precisely, if a classification up to isomorphism of a collection of objects
would imply the classification of finite-dimensional modules over the free 
$k$-algebra $k\la x,y\ra$, then the problem is considered {\em wild}; otherwise it is 
{\em tame}.  
While $k\la x,y\ra$-modules may sound obscure,
their classification would imply that of all finite-dimensional modules of
all finite-dimensional algebras -- clearly a wild problem in any sense of the term.  

As abelian groups and extraspecial groups are classified, these are natural examples of families
with a tame classification problem.
On the other hand, Vishnivetski{\u\i} showed in~\cite{Vish:1} that  
the classification problem for the groups in Theorem~\ref{thm:main1} is tame, but
not by giving a classification (this is likely a very hard problem).
If the constraints on the groups in Theorem~\ref{thm:main1} are relaxed in mild ways, their classification 
problem becomes wild. For instance, the problem is wild
if the exponent $p$ condition is removed~\cite{Sergeichuk}. In another direction, 
the classification of exponent $p$ groups with central commutator subgroup  
$\mathbb{Z}_p\times\mathbb{Z}_p\times\mathbb{Z}_p$ is also wild~\citelist{\cite{BLS}\cite{BDLST}}.

It is not known if groups of genus $2$ over arbitrary fields are wild or tame, and we certainly
do not come close to a full classification here. Crucial to our algorithms, however, is a
canonical representation of the genus $2$ groups that are not central products of 
proper nontrivial subgroups.  We prove the following result in Section~\ref{sec:genus2}.

\begin{thm}
\label{thm:indecomps}
Let $G$ be a centrally indecomposable $p$-group of genus $2$ over a field $k$. 
Then $G$ is isoclinic to
one of the following two types of groups:
\begin{enumerate}[(i)]
\item a central quotient of a Heisenberg group,
\begin{align*}
	H &  =\left\{ \begin{bmatrix}
	1 & e & w \\ 0 & 1 & f \\ 0 & 0 & 1
	\end{bmatrix} : \begin{array}{c} e,f,w\in k[x]/(a(x)^c),\\ a(x) \textnormal{ irreducible } \end{array}\right\},
\end{align*} by a subgroup $L\leq [H,H]\cong k^{m}$, $m=c\deg m(x)$, which is a $k$-subspace of codimension $2$; or

\item the matrix group
\begin{align*}
	H^{\flat} & = 
	\left\{
	\left[\begin{array}{c|c|c}
		I_2 & \begin{array}{cccc} 
			e_1 & \cdots & e_m & 0 \\
			0 & e_1 & \cdots & e_m
		\end{array} & 
		\begin{array}{c} z_1 \\ z_2 \end{array}\\
		\hline		
		 & I_{m+1} & \begin{array}{c} f_0\\ \vdots \\ f_{m}\end{array}\\
		 \hline
		 & & 1
	\end{array}\right] : e_i, f_j,w_{\ell} \in k\right\}.
\end{align*} 
\end{enumerate}
\end{thm}

The relationship between complexity of classification and computational complexity 
is examined in a recent work of Lipyanski and Vanetik~\cite{LV}. They provide, in particular, 
a summary of known connections between wild and tame problems and graph isomorphism.


\subsection{Taming the groups of genus 2}
\label{subsec:prop-overview}
The reader familiar with $p$-groups and algorithms may already have noticed the intimate
connection, which will be further elucidated in Sections~\ref{sec:bimaps} and~\ref{sec:genus2},
between groups of genus 2 and pairs $\{\Phi_1,\Phi_2\}$ of alternating
forms over a finite field $\mathbb{F}_q$. Indeed, we use the classification of 
such pairs by Bond~\cite{Bond} and Scharlau~\cite{Scharlau} 
(who themselves use an earlier classification of pairs of matrices --
or {\em Kronecker modules} as they are known -- by
Kronecker and Dieudonn\'{e}~\cite{Dieudonne}).  However, we also exploit a recently
discovered Galois connection between adjoints of bilinear maps and tensor 
products~\citelist{\cite{Wilson:division}\cite{BW:autotopism}}
to prove Theorem~\ref{thm:indecomps}.

By itself, Theorem~\ref{thm:indecomps} is not sufficiently powerful to
decide isomorphism among genus 2 groups (not even if we restrict to 
centrally indecomposable groups). 
There exist non-isomorphic groups of genus 2 whose centrally 
indecomposable factors are isomorphic (see Example~\ref{subsec:special-ex}). 
Hence, theorems of Krull-Remak-Schmidt type, upon
which the classification of Kronecker modules depends, simply do not
exist for groups of genus 2. Nevertheless, we prove a transitivity result
on fully-refined central decompositions of groups of genus 2 (Theorem~\ref{thm:transitive})
that serves as the foundation for an isomorphism invariant 
(strictly speaking, an {\em isoclinism} invariant) 
based on a generalization of the Pfaffian of an alternating form. 
The resulting characterization of isomorphism classes -- presented
in Theorem~\ref{thm:det-method} in terms of bilinear maps --
leads to an isomorphism test that is effective when $\mathbb{F}_q$ is small. 

When $\mathbb{F}_q$ is large, we use a general
technique for isomorphism testing in groups of exponent $p$-class 2.
Dubbed the {\em adjoint-tensor method}, this technique was
proposed in~\cite{BW:autotopism} by the first and third authors as
a means of bridging the gap between the generic but typically slow 
nilpotent quotient algorithms, and incredibly fast but highly specialized 
isomorphism tests such as the one in~\cite{LW}.
The adjoint-tensor method is presented in mildly restricted form in Section~\ref{sec:adjten}.
It requires the user to solve several problems -- such as  {\em algebra conjugacy}, {\em algebra normalizer},
and {\em subspace transporter} -- that are known in their general forms
to be hard.  With some considerable effort, however, the constraints imposed by
the genus 2 assumption may be exploited to overcome each obstacle,
and ultimately to produce an efficient test for isomorphism.
The details of the test, which include effective methods for
computing with certain quotients of the notorious Nottingham group,
are presented in Sections~\ref{sec:iso-ind-genus2} through~\ref{sec:iso-auto}.

\section{Nilpotent Groups and Bimaps}
\label{sec:bimaps}

We describe the relationship between groups of nilpotence class $2$ and bilinear maps.  
This has a long history going back to  work of Brahana and Baer in the 1930's. 
Henceforth, all groups are finite.

\subsection{Bimaps}
\label{subsec:bimaps}
Let $k$ be a commutative ring, and $U,V,W$ (left) $k$-modules.
A {\em $k$-bilinear map}, which we abbreviate to {\em $k$-bimap}, is a function
$\circ \colon U\times V\bmto W$ such that, for all $u,u'\in U$, $v,v'\in V$, 
and $\alpha\in k$, 
\begin{align}
\label{eq:bimap}
(u+\alpha u') \circ  v &= u\circ v + \alpha( u'\circ v ), \\
u \circ (v + \alpha v') &= u\circ v + \alpha( u\circ v').
\end{align}
The \emph{radicals} of $\circ$ are $U^{\bot}=\{ v\in V \colon U\circ v = 0 \}$, 
$V^{\top}=\{u\in U: u\circ V=0\}$, and $W^{+}=W/(U\circ V)$.
We say $\circ$ is {\em fully nondegenerate} if all three radicals are trivial.
If $U=V$ and $v\circ v = 0$ for all $v$, then $\circ$ is {\em alternating}.
We reserve the use of $U$, $V$ and $W$ for these three variables of a bimap
and write $U_{\circ}$, $U_{\bullet}$, and so forth if we need to distinguish between 
these components for separate bimaps $\circ$, $\bullet$.

A {\em homotopism} between bimaps $\bullet\colon U_{\bullet}\times V_{\bullet}\bmto W_{\bullet}$ and $\circ \colon U_{\circ}\times V_{\circ}\bmto W_{\circ}$ 
is a triple $f=(f_U,f_V;f_W)\in \hom(U_{\bullet},U_{\circ})\times \hom(V_{\bullet},V_{\circ})\times \hom(W_{\bullet},V_{\circ})$ such that
\begin{align} 
\label{eq:pseudo-isometry}
(\forall & u\in U_{\bullet},\forall v \in V_{\bullet}) &
 u f \circ v f = (u\bullet v)f.
\end{align}
When working with such a triple of maps, writing $uf$ for $u\in U$ means $uf_U$,
whereas $vf$ for $v\in V$ means $vf_V$, and so on. This occurs in one or two 
other places in the paper.
Bimaps with homotopisms form a natural category \cite{Wilson:division}.
A homotopism in which all maps are invertible is an {\em isotopism}.
We typically work here with alternating bimaps, and for such bimaps
we shall further insist that $\phi_U=\phi_V$ and refer to an isotopism 
between $\bullet$ and $\circ$ as a {\em pseudo-isometry}. 

When we need to describe a bimap in an example -- or for computation -- we 
do so via matrices. Fix generating sets $X,Y,Z$ for $U,V,W$, respectively, as $k$-modules. 
For $x\in X$, $y\in Y$, there exist $\alpha_{xyz}\in k~(z\in Z)$ 
such that
\begin{align*}
	 x\circ y &= \sum_{z\in Z} \alpha_{xyz} z.
\end{align*}
The scalars $\alpha_{xyz}$ are called
\emph{structure constants}
of $\circ$ relative to $X,Y,Z$.  
When $k$ is a field, these constants are uniquely determined by the choices of 
$X,Y,Z$ and we record the data using matrices $\Phi_z=[[\alpha_{xyz}]]$,
where $z\in Z$ and each $\Phi_z$ is an $|X|\times |Y|$ matrix. 
When $\circ$ is alternating, each $\Phi_z$ represents an alternating form on $U=V$, and
$\left\{\Phi_z\colon z\in Z\right\}$ is commonly known as a {\em system of forms}~\cites{B-F,BO}.


\subsection{Isoclinism and isomorphism of groups}
One can associate to each nilpotent group $G$ of class 2 an alternating bimap $\circ_G$.
Equivalence of such bimaps up to pseudo-isometry corresponds to an equivalence
of groups that is in general weaker than isomorphism. This equivalence was introduced
by Philip Hall and is known as {\em isoclinism}. The relationship between 
isoclinism and isomorphism for groups
is akin to that between homotopy and homeomorphism for topological spaces.
\medskip

Commutation in a group is a function 
$[,]\colon G\times G\to G$ whose image is not, in general, a subgroup of $G$. 
However, the subgroup generated by this image is the 
{\em commutator subgroup} and is denoted $[G,G]$ or $G'$.
Commutation is also not a homomorphism and hence has no kernel.
However, the {\em center} of $G$, namely
$Z(G)=\{g\in G\colon [G,g]=[g,G]=1\}$  consists of those elements
that do not influence the outcome of commutation.  
Thus,
we reduce $[,]$ to the \emph{commutation word map}
\[
\begin{array}{rccc}
	\circ_G \colon  & G/Z(G) \times G/Z(G) & \to &  G'\\ 
	& (\bar{x},\bar{y} ) & \mapsto &  [x,y],
\end{array}
\]
where $\bar{x}$ denotes the coset $xZ(G)$.
Comparing groups $G$ and $H$ only up to their commutation structures is therefore comparing
the functions
$\circ_G$ and $\circ_H$.  Doing so requires homomorphisms $f\colon G/Z(G)\to H/Z(H)$
and $\hat{f}\colon G'\to H'$ 
such that
\begin{align}
	(\forall & x,y\in G) & \bar{x}f\circ_H \bar{y}f & = (\bar{x}\circ_G\bar{y})\hat{f}.
\end{align}
The pair $(f,\hat{f})$ is a \emph{homoclinism} and, if the pair is invertible, it is an \emph{isoclinism}.  
Groups up to homoclinism
form a category with all the expected properties.

\begin{remark}
One can replace commutation with any word $w(x_1,\dots,x_n)$, producing a word map
$w\colon (G/w^*(G))^n\to w(G)$ where $w(G)$ is the \emph{verbal subgroup} generated by all evaluations of $w$
and $w^*(G)$ is the \emph{marginal subgroup} consisting of elements that do not influence the evalations
of $w$.  The associated equivalence up to $w$ is known as a {\em $w$-isologism}.  
Isoclinism is therefore the special case of $[x,y]$-isologism.
\end{remark}

In \cite{Baer:correspondence}, 
Baer established a fundamental correspondence for class 2 nilpotent groups
that may already be evident from the foregoing. 
 
\begin{thm}[Baer correspondence]
\label{thm:baer}
If $[G,G]\leq Z(G)$ then $\circ_G$ is a fully nondegenerate alternating $\mathbb{Z}$-bimap.  
Further, two groups $G$ and $H$ of nilpotence class $2$ are isoclinic if, and only if, 
$\circ_G$ and $\circ_H$ are pseudo-isometric.
\end{thm}

The next crucial observation follows from the Universal Coefficients Theorem 
applied to group cohomology. (Direct proofs
are also known; see, for example, \cite{Wilson:unique-cent}*{Proposition 3.10}.)

\begin{prop}
\label{prop:isoclinic-isomorphic}
Two $p$-groups of nilpotence class $2$ and exponent $p$ are isoclinic if, and only if, they are isomorphic.
\end{prop}

In view of Theorem~\ref{thm:baer} and Proposition~\ref{prop:isoclinic-isomorphic} it makes sense,
for a fixed alternating bimap $\circ\colon V\times V\bmto W$, to consider the 
{\em pseudo-isometry} group
\begin{align}
\label{eq:pseudo-isometry-group}
\pseudo(\circ) &= \{ (\phi,\hat{\phi}) \in \Aut(V)\times \Aut(W) \colon 
\forall u,v\in V, u\phi\circ v\phi = (u\circ v)\hat{\phi}\}.
\end{align}

The following observation, which reduces questions of isomorphisms and automorphisms of groups 
to ones about bimaps, is proved in \cite{Wilson:unique-cent}*{Proposition 3.8}.

\begin{prop}
\label{prop:autoclinism}
If $G$ is nilpotent of class $2$, and ${\rm Aut}(G)$ its group of automorphisms, then
\begin{align*}
\xymatrix{
	1 \ar[r] &  C_{\Aut(G)}(Z(G)) \ar[r]^{\iota} & 	\Aut(G)\ar[r]^{\pi} & \pseudo(\circ_G)
}
\end{align*}
is an exact sequence.
If $G$ has exponent $p$ then $\pi$ is surjective.
\end{prop}

The proof of Proposition~\ref{prop:autoclinism} shows, in fact, that the group of {\em autoclinisms} 
of $G$ coincides with $\pseudo(\circ_G)$.
\medskip

Finally, for a fixed alternating bimap $\circ\colon V\times V\bmto W$ the {\em isometry} group is
\begin{align}
\label{eq:isometry-group}
\isom(\circ) = \{ \phi\in \Aut(V) \colon  \forall u,v\in V, u\phi\circ v\phi = u\circ v \}.
\end{align}
This is the kernel of the restriction of $\pseudo(\circ)$ to $W$.  

\subsection{Computational models for groups}
\label{sec:models}
Efficient algorithms exist to determine crucial information about groups. Details and proofs can be found in \citelist{\cite{Seress}\cite{HoltEO}}.
The meaning of efficiency depends on how groups are specified for computation.
\smallskip

The pioneering work of Sims, Cannon, and Neub\"user in the 1960s and 
1970s led to the standard models of computation that we use today.  
It is most common to specify general finite groups by small sets of generators 
(matrices over finite fields or  permutations of a finite set). Special classes of 
groups admit certain types of structured presentations as feasible computational models. 
For example, polycyclic presentations are often used for computations with solvable groups. 
Algorithms for $p$-groups should certainly apply to these specialized models but we caution that the 
complexity of multiplication can be exponential~\citelist{\cite{collection}\cite{deep-thought}}.
The notion of a ``black-box'' group was introduced by Baba\'i and Szemer\'edi 
as a means of stripping away information specific to the particular representation,
and thereby forcing algorithms to deal only with the algebraic structure of the group.

We elect, here, to assume that groups are given as quotients of representations of groups.   
This input model lies between groups given by concrete representations and
groups given in an abstract black box model. 
We mean by this that quotient groups do not act on the underlying 
set (or vector space) so one cannot naturally work with orbits (or modules)
of the given group. On the other hand, the parent group does act,  and hence provides some  
access to representation specific methodology.
This led Kantor and Luks~\cite{KL:quo} to propose the 
``quotient group thesis": {\em problems soluble in polynomial time for permutation groups
ought also to be soluble in polynomial time for quotients of such groups}
(although usually with rather more sophisticated methods).
Of critical importance to us is that often $p$-groups cannot be represented faithfully as permutations 
on small sets.  As an illustration, P. Neumann showed that 
extraspecial groups of order $2_+^{2m+1}$ have no faithful permutation 
representation on fewer than $2^m$ points, yet they
are {\em quotients} of $D_8^m$ represented on $4m$ points. We show
in Proposition~\ref{prop:degree} that this phenomenon 
also occurs for groups of genus $2$.  The inclusion of quotients 
indicates that our algorithms are prepared to 
consider groups whose order is exponential in the input size.
\smallskip

Whichever computational model we wish to work with, we require an analogue for that model of the 
following fundamental result.

 \begin{prop}
\label{prop:basic-algs}
Given a group $G$ as a quotient of a permutation group, in polynomial time one can
\begin{enumerate}[(i)]
\item find $|G|$,
\item write $x$ as a word over $x_1,\dots, x_c\in G$ or prove $x\not\in\langle x_1,\dots,x_c\rangle$,
\item find generators for $Z(G)$ and for $G'$,
\item decide if $G$ is nilpotent of class $2$, and
\item if $G$ is nilpotent of class 2, construct a system of forms for $\circ_G$.
\end{enumerate}
\end{prop}
\begin{proof}
For (i)--(iv) see \cite{Seress}*{Chapter 6} and \cite{KL:quo}.
For (v), fix bases $\{x_1,\dots,x_d\}$ and $\{w_1,\ldots,w_e\}$ for the abelian groups $G/Z(G)$ and $G'$,
respectively. The structure constants for the associated system of forms are  obtained by writing
each $[x_i,x_j]$ as a vector relative to $\{w_1,\ldots,w_e\}$.
\end{proof}


We shall state and prove various results for bimaps that require us to work with large fields.
We therefore allow ourselves to factor polynomials using randomized Las Vegas 
polynomial-time factorization algorithms.
(A Las Vegas algorithm always returns a correct result but with small, user prescribed, probability 
reports failure.)  Such methods can always be ``derandomized" whenever the characteristic $p$ is 
bounded by the input size -- as is the case with permutation group quotients.  
We refer the reader to~\cite{vzG} for further information on these matters.


\section{Groups of Genus $2$}
\label{sec:genus2}

In this section we propose an integer metric for the ``complexity" of a nilpotent group.  
Inspired by an analogous metric introduced by Knebelman to measure the
complexity of Lie algebras, we call this number the {\em genus} of a group.
The broader class of groups underlying
Theorem~\ref{thm:main1} are the groups of genus 2.

\subsection{The centroid and genus of a group}
Let $k$ be a commutative ring, and $\circ\colon U\times V\bmto W$ a $k$-bimap.  
The \emph{centroid} of $\circ$ is the largest 
ring, $C$, over which $\circ$ is $C$-bilinear, namely
\begin{align*}
	C({\circ}) & = \{ \sigma\in \End(U)\times \End(V)\times \End(W) \colon  \forall u, \forall v,
		(u\sigma) \circ v= (u\circ v)\sigma=u\circ (v\sigma)\}.
\end{align*}
It is understood that $\sigma\in C(\circ)$ acts naturally on $U$, $V$, and $W$ but 
we can write $\sigma=(\sigma_U, \sigma_V;\sigma_W)$ if we wish
to clarify the action on the individual $k$-modules. 
The explicit definition of $C(\circ)$ makes it clear that this ring may be obtained
as the solution of a system of linear equations.
If $\circ$ is fully nondegenerate 
-- as is the case with the commutation bimap of a group -- then $C(\circ)$ is commutative.  
The following connection between centroids and direct products was proved
in~\cite{Wilson:RemakI}*{Section 6}.

\begin{thm}
\label{thm:centroid-direct}
For a finite nilpotent group $G$ of class $2$, $G$ is isoclinic to a direct product of proper 
nontrivial subgroups if, and only if,
the centroid $C(\circ_G)$ is a direct product of proper subrings.
\end{thm}

Being concerned with questions of isomorphism, we focus on the directly indecomposable 
factors of a group -- thus, by Theorem~\ref{thm:centroid-direct}, $C=C(\circ_G)$
is a local ring.  Thus, if $J=J(C)$ is the Jacobson radical of $C$,
$W/WJ$ is a vector space over the residue field $C/J$, and we define the {\em rank}
of $W$ to be the dimension of this space.

\begin{defn}
\label{def:genus}
Let $G$ be a nilpotent group of class $2$. Then $G$ is isoclinic to a direct product 
$H_1\times \cdots \times H_s$ of directly indecomposable
groups.  The \emph{genus} of $G$ is maximum rank of $[H_i,H_i]$ as a 
$C(\circ_{H_i})$-module.
\end{defn}

The concept of genus arose first in Knebelman's attempts to
classify Lie algebras and general nonassociative algebras~\cite{Knebelman:genus} .
He observed that when the dimension of a Lie $k$-algebra $L$ 
was close the minimum number, $d(L)$, of generators, there are relatively few variable relations. 
Accordingly, he proposed that algebras of low {\em genus} -- which he defined as $\dim L-d(L)$ -- 
should be  easier to classify.  For instance, if $L$ is abelian then $\dim L-d(L)=0$,
and if $L$ is a Heisenberg Lie algebra then $\dim L-d(L)=1$.  Later, 
Bond tackled the classification of  Lie algebras of genus 2, and reduced 
the problem to the class $2$ nilpotent Lie algebras of genus $2$~\cite{Bond}.    
The latter problem remains very difficult.
In fact the classification of 6-dimensional 
Lie algebras has only recently been completed~\citelist{\cite{Morozov}\cite{CdGS}}, and
the nilpotent genus $2$ cases are the most involved.

\subsection{Some groups of low genus} 
\label{subsec:first-examples}
To reveal some important subtleties in the definition of genus,
and to provide concrete examples of the groups
we propose to study, we introduce some groups
of genus 1 and genus 2.

\begin{enumerate}[(a)]
\item Every group with cyclic central commutator subgroup has genus $1$.
For such $G$ with
$\circ_G\colon G/Z(G)\times G/Z(G)\bmto \mathbb{Z}_m$ we have
$C(\circ_G)=\mathbb{Z}_m=\mathbb{Z}_{p_1^{e_1}}\oplus \cdots \oplus \mathbb{Z}_{p_s^{e_s}}$,
with each $p_i$ a distinct prime.    As $G$ is a direct product of its Sylow subgroups, 
we need only the maximum genus when restricted to each $p_i$.  As each $\mathbb{Z}_{p_i^{e_i}}$ is
cyclic, the genus of each Sylow subgroup is $1$.  

\item Any group with central commutator subgroup isomorphic to 
$\mathbb{Z}_m\times \mathbb{Z}_n$ has genus at most $2$.
Let $G$ be such a group.
If $(m,n)=1$, then $\mathbb{Z}_m\times \mathbb{Z}_n$ is cyclic and
$G$ has genus 1.  Else, $G$ is again a product of Sylow subgroups. 
Let $P$
be a Sylow $p$-subgroup of $G$ of largest genus.  
We may assume 
$P'\cong \mathbb{Z}_{p^e}\times \mathbb{Z}_{p^f}$, $e\geq f\geq 1$. Either 
$C(\circ_P)\cong\mathbb{Z}_{p^e}$ (in which case $P$ is genus 2), or $C(\circ_P)$ 
is not local and $P$ is isoclinic to a nontrivial direct product (so $P$ is genus 1).  

\item Fix any commutative Artinian ring $K$, and consider the
Heisenberg groups 
\begin{align}\label{eq:Heisenberg}
	H_m(K) & =\left\{\begin{bmatrix} 1 & u & s \\ 0 & I_m & v^{{\rm tr}} \\ 
	0 & 0 & 1 \end{bmatrix} : u,v\in K^m, s\in K\right\}.
\end{align}
If $K=K_1\oplus \cdots \oplus K_s$ be a decomposition of $K$ into local rings,  then
$H_m(K)\cong H_m(K_1)\times \cdots \times H_m(K_s)$, so the genus is the maximum genus of any 
$H_m(K_i)$. The bimap of $H_m(K_i)$ is simply the alternating nondegenerate form 
$K_i^{2m}\times K_i^{2m}\bmto K_i$ having $K_i$ as centroid.  
Since $K_i$ is commutative and local, 
$K_i/J(K_i)$ is a field, so $H_m(K_i)$ has genus $1$. Hence, all Heisenberg 
groups have genus 1.
\end{enumerate}

While all of these examples are somewhat elementary, 
from the point of view of classification
we have already entered turbulent waters. For instance,
classifying the groups with cyclic central commutator in part (a)  has 
taken the combined work of several authors including Leong \cite{Leong}, Finogenov \cite{Fin} and 
Blackburn \cite{Blackburn}.
The Heisenberg groups in (c) were only recently characterized in abstract terms
(with no a priori knowledge of $m$ or $K$)
for the case when $K$ is a field~\cite{LW}*{Theorem 3.1}.  
(Our results extend that characterization to the case when $K$ is a cyclic algebra.)
\medskip

It may surprise the reader that groups seeming to have genus $g>1$ are in fact genus 1.
For example, $H_1(\mathbb{F}_{p^g})$ is a
group whose central commutator is isomorphic to $\mathbb{Z}_p^g$, so it would seem
that one can easily build a group of high genus. However, the centroid recovers field
structure, and viewed as a vector space over the centroid the commutator is 1-dimensional.  
Similarly, a direct product $H_1(\mathbb{F}_p)^g$ again has commutator
$\mathbb{Z}_p^g$. Via Theorem~\ref{thm:centroid-direct} and Definition~\ref{def:genus}, however,
those examples are also genus 1.  
Note, moreover, that our definition of centroid is invariant under extensions: 
if $G$ is a group of genus $g$ over $K$, and  
$H$ is a group such that $\circ_H$ is the tensor of $\circ_G$ with a field 
extension $E$ of $K$, then $H$ has genus $g$ over $E$.

\subsection{Central decompositions, hyperbolic pairs, and adjoints}
\label{sec:decomps-genus2}
The groups of genus $2$ admit two important decompositions.  
The first decomposes the group as a central product of subgroups,
and the second as a product of two abelian normal subgroups whose
intersection is central. 
We shall make essential use of both types of decomposition in our algorithms,
so we now introduce them and give characterizations that facilitate
effective computation.

\begin{defn}
\label{defn:cent-decomp}
A {\em central decomposition} of a group $G$ is a set, $\mathcal{H}$, of subgroups generating 
$G$ such that for $H\in\mathcal{H}$, $G\neq \langle \mathcal{H}-\{H\}\rangle$ and
$[H,\langle \mathcal{H}-\{H\}\rangle]=1$.  We say that $G$ is {\em centrally indecomposable} 
if $\{G\}$ is its only central decomposition. 
\end{defn}

A detailed treatment of central decompositions of $p$-groups is the subject of \cite{Wilson:unique-cent},
and we shall use some of the results therein. The second decomposition we need mimics hyperbolic pairs
in the sense of symplectic geometry.  It was introduced in \cite{LW}*{Section 6} to work with $2$-groups, but we
use it here for arbitrary $p$-groups.

\begin{defn}
\label{defn:hyp-pair}
A {\em hyperbolic pair} for a group $G$ is a pair $M,N$ of normal abelian subgroups of $G$ such that 
$G=MN$ and $M\cap N\leq Z(G)$.
\end{defn}

Both central and hyperbolic decompositions may be obtained from a ring that is easily computed from $\circ_G$,
namely the ring of adjoints.
In a similar vein to our definition of centroid, we introduce the \emph{adjoint ring}, $A(\circ)$, of a bimap
$\circ\colon U\times V\bmto W$ as the largest ring, $A$,
over which $\circ$
factors through $U\otimes_A V$, namely
\begin{align*}
	A(\circ) & = \{ \mu\in \End(U)\times\End(V)^{{\rm op}} : \forall u\forall v, u\mu\circ v = u\circ \mu v\}.
\end{align*}
Again, $A(\circ)$ may be obtained as the solution of a system of 
linear equations~\citelist{\cite{Wilson:find-central}\cite{BW:isom}\cite{BW:slope}}. 
As $\End(V)^{{\rm op}}$ suggests, we find it convenient work with the opposite ring 
in the second component -- thus $A$ acts on $U$ on the right but on $V$ on the left.
If we need to clarify the action we write $u\mu =uL_\mu$ and $\mu v=vR_{\mu}$.

If $\circ\colon V\times V\bmto W$ is a nondegenerate, alternating bimap,
then $A(\circ)$ is faithfully represented
in $\End(V)$ and in $\End(V)^{\rm op}$.  This endows $A(\circ)$ with a
natural anti-isomorphism interchanging $R_\mu$ and $L_\mu$,
giving it the structure of a $*$-ring.  
The connections to central decompositions and hyperbolic pairs 
come from the existence of certain types if idempotents in this $*$-ring.
We say that an idempotent, $e$, in $A(\circ)$ is {\em self-adjoint} if $e^*=e$,
and {\em hyperbolic} of $e^*=1-e$.
Recall that idempotents $e,f$ in a ring are {\em orthogonal} if $ef=0=fe$.

\begin{lemma}
\label{lem:selfadj-idemp}
A finite nilpotent group, $G$, of nilpotence class 2, has
\begin{enumerate}[(i)]
\item  a central decomposition $\{H_1,\dots,H_s\}$ if, and only if,
$A(\circ_G)$ has a set $\{e_1,\dots,e_s\}$ of pairwise orthogonal, 
self-adjoint idempotents that sum to $1$, and
\item a hyperbolic pair if, and only if, 
$A(\circ_G)$ has hyperbolic idempotents.
\end{enumerate}
\end{lemma}

\begin{proof}
A proof of (i) may be found in~\cite{Wilson:unique-cent}*{Theorem 4.10}.

For (ii), let $Z=Z(G)$, and $\circ=\circ_G\colon G/Z\times G/Z\bmto G'$.
Suppose $M,N$ is a hyperbolic pair for $G$, and put 
$E=MZ/Z$ and $F=NZ/Z$.  
Since $G=MN$ and $M\cap N \leq Z(G)$ we have $V=G/Z=E\oplus F$.  
Let $e$ denote the projection idempotent onto $E$ with kernel $F$.
Hence, $1-e$ is the projection idempotent onto $F$ with kernel $E$.
As $M$ and $N$ are
abelian, note $E\circ E=0=F\circ F$, so  
for all $u,v\in V$, 
\begin{align*}
ue\circ ev=0=u(1-e)\circ(1-e)v,
\end{align*}
and we have  
\begin{align*}
	ue\circ v & = ue\circ ev+ue\circ (1-e)v \\
	               & =u\circ (1-e)v -u(1-e)\circ (1-e) v \\
	               & = u\circ (1-e)v.
\end{align*}
In particular, $e\in A(\circ)$, and $e^*=1-e$.  

Conversely, observe that if $e\in A(\circ)$ and 
$e^*=1-e$, then $V=Ve\oplus V(1-e)$ and
$Ve\circ e V=Ve(1-e)\circ V=0$.  
Hence, $M=\{ g\in G \colon  (gZ)e=gZ\}$ and 
$N=\{g\in G \colon  gZ(1-e)=gZ\}$ is a hyperbolic pair
for $G$.
\end{proof}

Lemma~\ref{lem:selfadj-idemp} provides a tool to locate central decompositions and hyperbolic pairs.

\begin{thm}
\label{thm:cent-hyp-alg}
There are polynomial-time algorithms for each of the following:
\begin{enumerate}[(i)]
\item construct a fully-refined central decomposition of a given finite $p$-group; and
\item decide if a given $p$-group has a hyperbolic pair and construct one if it does.
\end{enumerate}
\end{thm}

\begin{proof}
A proof of (i) may be found in~\cite{Wilson:find-central}*{Theorem 1.1}.  

The proof for (ii) is similar so we just give a sketch. Let $G$ be the given $p$-group,
and $\circ=\circ_G\colon G/Z(G)\times G/Z(G)\bmto G'$.
Recall that we can compute generators for $A(\circ)$ as the solution of a system
of equations. Hence, by
Lemma~\ref{lem:selfadj-idemp}, it suffices to find an  idempotent 
$e\in A(\circ)$ such that $e^*=1-e$.  

Using \citelist{\cite{Wilson:unique-cent}\cite{BW:isom}\cite{BW:slope}} we begin
by constructing the Jacobson radical, $J$, of $A$, and then decomposing $A/J$
as a sum $S_1\oplus\ldots\oplus S_m$ of $*$-simple ideals.
Each $S_i$ is isomorphic to the
adjoint ring of a nondegenerate alternating, symmetric, or Hermitian form 
(where in the Hermitian case we permit a degenerate field extension
$K\oplus K$ -- also called exchange); see~\cite{Wilson:find-central}*{Section 5}. 

In a $*$-simple ring, an idempotent
$e$ with $e^*=1-e$ coincides with a decomposition of the associated 
form into a pair of totally isotropic subspaces, which are readily computed
using a Gram-Schmidt type algorithm.  
Thus, within each $S_i$ we locate an idempotent
$\hat{e}_i$ with $\hat{e}_i^*=1-\hat{e}_i$.  Let $\hat{e}=\sum_i \hat{e}_i$ and use the idempotent 
lifting formula in~\cite{Wilson:unique-cent}*{Section 5.4} to lift $\hat{e}\in A/J$
to an idempotent $e\in A$ with $e^*=1-e$.
\end{proof}

We remark that one can lift idempotents more efficiently when $p$ is odd 
by computing a $*$-invariant semisimple complement
to the radical, thereby reducing the problem to the semisimple $*$-rings~\cite{BW:isom}.

\subsection{The centrally indecomposable groups of genus 2}
\label{subsec:index-theory}
We focus now on the centrally indecomposable groups of genus 2. 
Our immediate goal is to classify the adjoint rings of the commutation bimap of these groups.
The ultimate goal is to prove Theorem~\ref{thm:indecomps}, but this must wait until Section~\ref{subsec:indecomp-proof}.

We begin with a classification by Kronecker and Dieudonn\'{e}~\cite{Dieudonne} of pairs of matrices,
which later led to classifications of pairs of forms by Scharlau~\cite{Scharlau}.  Independently --
and prior to Scharlau -- Bond \cite{Bond}*{p. 608} applied the same treatment to attempt to classify
nilpotent Lie algebras of genus $2$.
\smallskip

The following fundamental result is folklore (see, for example,~\cite{GG}*{Section 1}).

\begin{lemma}
\label{lem:isotropics}
If $\{\Phi_1,\Phi_2\}$ is a pair of alternating forms on a 
finite-dimensional vector space $V$, then
there is a decomposition $V=E\oplus F$ such that $E$ and $F$ are totally 
isotropic with respect to both forms.
\end{lemma}

Let $\Psi_1,\Psi_2$ be a pair of $c\times d$ matrices with entries in a field $k$.
As transformations from $k^c$ to $k^d$ we say that
a $\Psi_1,\Psi_2$ is {\em decomposable} if we can find a bases for $k^c$ and $k^d$ 
such that $\Psi_1=\left[ \begin{smallmatrix} \Psi_{11} & 0 \\ 0 & \Psi_{12} \end{smallmatrix} \right]$ and
$\Psi_2=\left[ \begin{smallmatrix} \Psi_{21} & 0 \\ 0 & \Psi_{22} \end{smallmatrix} \right]$; otherwise the pair
is {\em indecomposable}. 
Indecomposable pairs are classified in the following classical result.
(An algorithm for this result is given in Section~\ref{subsec:standard}.)

\begin{thm}[Kronecker-Dieudonn{\'e} \cite{Dieudonne}]
\label{thm:Dieudonne}
If $\Psi_1,\Psi_2$ is an indecomposable pair of matrices with entries in a field $k$, then
one of the following holds:
\begin{enumerate}[(i)]
\item $\Psi_1,\Psi_2\in\M_d(k)$ and there are
bases such that $\Psi_1=I_d$ and $\Psi_2=C(a(x))$, 
where $a(x)$ is a power of an irreducible polynomial and $C(a(x))$ its companion matrix; or
\item $\Psi_1,\Psi_2\in\M_{d,d+1}(k)$ and there are bases such that $\Psi_1=[I_d | 0]$ and 
$\Psi_2= [0 | I_d]$.
\end{enumerate}
\end{thm}

Note that this result asserts a canonical description of indecomposable pairs up to the action
$\{\Psi_1,\Psi_2\}\mapsto \{X\Psi_1 Y, X\Psi_2 Y\}$.  If we constrain the problem so that $X=Y^{-1}$ (and hence 
that the matrices are square) the classification problem becomes wild.  I.e. $k\langle x_1,x_2\rangle$ 
is represented in $\M_d(k)$ by $x_i\mapsto \Psi_i$ defining a $k\langle x_1,x_2\rangle$-module on $k^d$.  Conjugation of $\Psi_i$ by
$Y$ amounts to module isomorphism and this is the definition of wild representations.  Similarly, increasing from 
pairs of matrices to triples gives rise to another wild problem; cf. \citelist{\cite{BLS}\cite{BDLST}}.
\smallskip

We use Theorem~\ref{thm:Dieudonne} now to classify 
pairs of forms associated to centrally indecomposable $p$-groups of genus 2.

\begin{prop}\label{prop:slope-flat}
If $G$ is a centrally indecomposable $p$-group of genus $2$ over a field $k$, then $\circ_G\colon k^d\times k^d\bmto k^2$ 
is isometric to a bimap represented by a pair
\begin{align}
\Phi_1=\begin{bmatrix} 0 & \Psi_1 \\ -\Psi_1^{{\rm tr}} & 0 \end{bmatrix}~~\mbox{and}~~
\Phi_2 = \begin{bmatrix} 0 & \Psi_2 \\ -\Psi_2^{{\rm tr}} & 0 \end{bmatrix}, 
\end{align}
of alternating forms, where the pair $\{\Psi_1,\Psi_2\}$ is given by
Theorem~\ref{thm:Dieudonne}  part (i) or (ii) according to whether $d$ is even or odd,
respectively.
\end{prop}

\begin{proof}
Let $G$ be a centrally indecomposable group of genus $2$ over a field $k$.
Regard $V=G/Z(G)$ and $W=G'$ as $k$-spaces, so that $\dim_kW=2$, and
consider the $k$-bimap $\circ=\circ_G\colon V\times V\bmto W$.
As in Section~\ref{subsec:bimaps}, $\circ$ is represented by a pair 
$\{\Phi_1,\Phi_2\}$ of alternating forms over $k$. As $G$ is centrally indecomposable,
by Lemma~\ref{lem:selfadj-idemp}(i) $\{\Phi_1,\Phi_2\}$ is orthogonally indecomposable.
Further, by 
Lemma~\ref{lem:isotropics}, there is a decomposition $V=E\oplus F$ with
$E\circ E=0=F\circ F$. Thus, the restriction of $\circ$ to $E\times F$ 
yields an indecomposable pair of matrices (the ``corner blocks").   
Using appropriate basis changes in $E$ and $F$  we have
\begin{align}
\Phi_1=\begin{bmatrix} 0 & \Psi_1 \\ -\Psi_1^{{\rm tr}} & 0 \end{bmatrix}~~\mbox{and}~~
\Phi_2 = \begin{bmatrix} 0 & \Psi_2 \\ -\Psi_2^{{\rm tr}} & 0 \end{bmatrix}, 
\end{align}
where the pair $\{\Psi_1,\Psi_2\}$ is given by
Theorem~\ref{thm:Dieudonne}  part (i) or (ii) depending on 
whether $\dim_k V$ is even or odd,
respectively.
\end{proof}

The dichotomy in Proposition~\ref{prop:slope-flat} 
-- and its eventual incarnation in Theorem~\ref{thm:indecomps} --
is fundamental to our algorithm, and we introduce some helpful 
terminology from~\cite{BW:slope} for easy reference. 

\begin{defn}
A centrally indecomposable group, $G$, of genus 2 
is said to be {\em sloped} if it
is type (i), and {\em flat} if it is type (ii). 
We extend the appropriate notion of sloped and flat 
to the associated $k$-bimap, $\circ_G$.
(Note, if
the $k$-dimension of $G/\Phi(G)$ is even then $G$ is sloped, and otherwise it is flat.)
\end{defn}

We stress that Proposition~\ref{prop:slope-flat} is not a classification of centrally 
indecomposable groups of genus $2$, even if their exponent is $p$.  That would first require  
a classification of irreducible polynomials.  Secondly -- and much more troubling for our algorithms --
{\em pairs of forms
are not unique to a group of genus 2.}

\begin{ex}
Let $k=\mathbb{Z}_3$, and put $a(x)=x^2+1$ and $b(x)=x^2+x+2$.  
The Heisenberg groups $H(k[x]/(a(x)^2))$ and $H(k[x]/(b(x)^2))$ are isomorphic (they are both over $\mathbb{F}_9$) and centrally 
indecomposable of genus $2$, but there
are choices of generators for these groups where the associated pairs of forms are non-isometric.
E.g.: natural choices give rise to bimaps represented, respectively, by the pairs
\begin{align*}
 \left\{\begin{bmatrix} 0 & I_4\\ -I_4 & 0 \end{bmatrix},\begin{bmatrix} 0 & C(a(x)^2) \\ -C(a(x)^2)^{{\rm tr}} & 0 \end{bmatrix}\right\},\\
  \left\{\begin{bmatrix} 0 & I_4\\ -I_4 & 0 \end{bmatrix},\begin{bmatrix} 0 & C(b(x)^2) \\ -C(b(x)^2)^{{\rm tr}} & 0 \end{bmatrix}\right\}.
\end{align*}
\end{ex}

Another way in which choice of generators removes a canonical relationship to Kronecker type arguments is
seen by constructing groups of genus $2$ as {\em quotients} of Heisenberg groups.

\begin{ex}\label{ex:strange}
Let $k=\mathbb{Z}_3$, and put $a_1(x)=x^4+x^3+x^2+1$, $a_2(x)=x^4+2x^2+2$, and $a_3(x)=x^4+x^3+2x+1$.  
Set $H_i=H(k[x]/(a_i(x)))$.  As each $a_i(x)$ is irreducible $H_1\cong H_2\cong H_3$.  Now, with respect to the natural basis $\{1,x,x^2,x^3\}$,
define 
\begin{align*}
	M_i & = \left\{\begin{bmatrix} 1 & 0 & a+bx \\ 0 & 1 & 0 \\ 0 & 0 & 1 \end{bmatrix} : a,b\in k\right\}\leq H_i.
\end{align*} 
Evidently, 
$G_i=H_i/M_i$ has genus $2$ over $k$ and is centrally indecomposable.  
However, $G_1\not\cong G_2$ while $G_2\cong G_3$ (there is an isomorphism induced by sending $x^2\mapsto 2x^2+x^3$ and $x^3\mapsto x^2$).
That $G_1\not\cong G_2$ can be settled either exhaustively or by apply the algorithm of Theorem~\ref{thm:main1}.
\end{ex}

We will soon need the following consequence of Proposition~\ref{prop:slope-flat}.

\begin{cor}
\label{cor:types}
Let $G$ be a centrally indecomposable $p$-group of genus $2$ over a field $k$, 
$A=A(\circ_G)$ its ring of adjoints, and $J=J(A)$ the Jacobson radical of $A$.
\begin{enumerate}[(i)]
\item $G$ is sloped if, and only if, $A/J\cong \M_2(K)$, where $K/k$ a field extension 
and the induced involution on $A/J$ is $\left[\begin{smallmatrix} a & b \\ c & d \end{smallmatrix}\right]\mapsto \left[\begin{smallmatrix} d & -b \\ -c & a \end{smallmatrix}\right]$.
\item $G$ is flat if, and only if, $A/J\cong K\oplus K$ with involution $(a,b)\mapsto (b,a)$.
\end{enumerate}
\end{cor}
\begin{proof}
By Lemma~\ref{lem:selfadj-idemp}(i), the only idempotents of $A$ with $e =e^*$ are $0$ and $1$.  Furthermore,
by Proposition~\ref{prop:slope-flat}, $G$ is hyperbolic and so $A$ has a hyperbolic idempotent $e^*=1-e$.
As $A$ is Artinian (in fact finite), idempotents lift over $J$, and by \cite{Wilson:find-central}*{Section 5.4} they lift retaining
the self-adjoint and hyperbolic relationships respectively.  Therefore, $A$ has such idempotents if, and only if, $A/J$ has
them.  Now we apply a classification due to Osborn (see \cite{Wilson:unique-cent}*{Theorem~4.26}) to assert
the only choices for $A/J$ are $M_2(K)$ and $K\oplus K$ together with the given involutions.  In the first case the
dimension of $G/Z(G)$ is even and hence corresponds to the sloped case.  In the second case the group is flat.
\end{proof}

\subsection{Uniqueness of orthogonal and hyperbolic decompositions}
\label{subsec:transitivity}
As we mentioned earlier, our algorithms for bimaps of genus 2 
will utilize both types of decomposition described in the previous
section. When we do so, we shall need to know that our 
particular choices are in fact generic. More precisely, we shall require the
following transitivity facts. (Recall, from Proposition~\ref{prop:isoclinic-isomorphic},
that isoclinisms coincide with isomorphisms for groups of exponent $p$.) 

\begin{thm}
\label{thm:transitive}
If $G$ be a finite $p$-group of genus 2, then the group
of autoclinisms of $G$ acts transitively on
\begin{enumerate}[(a)]
\item the set of fully-refined central decompositions of $G$, and
\item the set of hyperbolic pairs of $G$.
\end{enumerate}
In fact the subgroup of autoclinisms that centralize $Z(G)$ is transitive on both sets.
\end{thm}

\begin{proof}
For (a), refer to~\cite{Wilson:unique-cent}*{Theorem 6.6}. 
Corollary~\ref{cor:types} tells us that $\circ_G$ has no indecomposable summands 
of orthogonal type, and so $\isom(\circ_G)$ is transitive
on its fully-refined orthogonal decompositions.  

The proof of (b) is similar. By Witt's lemma, the isometries of a nondegenerate form
act transitively on the set of hyperbolic bases.  Then, using
involutions, one lifts this action over the radical.
\end{proof}

We stress that central product decompositions \emph{do not}, in general, 
possess such transitivity \cite{Wilson:unique-cent}*{Theorem~1.1(ii)}, so
groups of genus 2 are somewhat special in this regard.  Even so,
Theorem~\ref{thm:transitive} does not give rise to a theorem of Krull-Schmidt 
type~\cite{Wilson:unique-cent}*{Definition~2.6}. 
Indeed, as illustrated by the example below, identical sets of centrally indecomposable
groups may occur as fully-refined central decompositions of
non-isoclinic groups of genus 2. 
This hints at the difficulties in using central products within isomorphism tests.

\begin{ex}
\label{subsec:special-ex}
Let $k$ be any field, and $K=k(\omega)$ a quadratic extension of $k$.  
Put $H=H_1(k)\times H_1(k)\times H_1(K)$, a direct product of Heisenberg groups,
and let
\begin{align*}
		N_1 & =\left\langle
		\left(\begin{bmatrix} 1 & 0 & 1\\ 0 & 1 & 0 \\ 0 & 0 & 1 \end{bmatrix}, I_3, 
			\begin{bmatrix} 1 & 0 & -1 \\ 0 & 1 & 0 \\ 0 & 0 & 1 \end{bmatrix}\right),
\left(I_3, \begin{bmatrix} 1 & 0 & -1\\ 0 & 1 & 0 \\ 0 & 0 & 1 \end{bmatrix},
			\begin{bmatrix} 1 & 0 & -\omega \\ 0 & 1 & 0 \\ 0 & 0 & 1 \end{bmatrix}\right) 			
			\right\rangle\\
		N_2 & =\left\langle
		\left(\begin{bmatrix} 1 & 0 & 1\\ 0 & 1 & 0 \\ 0 & 0 & 1 \end{bmatrix}, 
		\begin{bmatrix} 1 & 0 & -1\\ 0 & 1 & 0 \\ 0 & 0 & 1 \end{bmatrix}, 
			I_3\right),
\left(I_3, \begin{bmatrix} 1 & 0 & 1\\ 0 & 1 & 0 \\ 0 & 0 & 1 \end{bmatrix},
\begin{bmatrix} 1 & 0 & -1\\ 0 & 1 & 0 \\ 0 & 0 & 1 \end{bmatrix}
			\right) 			
			\right\rangle.
\end{align*}
Then each $N_i$ is normal in $H$, and the groups $G_i=H/N_i$
have genus $2$ over $k$.  Moreover,
each group has a full-refined central decomposition
consisting of two copies of a group $X$, and one copy of a group $Y$,
yet $G_1$ and $G_2$ are non-isomorphic (in fact non-isoclinic).
For, if
$\omega$ has minimum polynomial $x^2-ax-b$.
Then, for $i=1,2$, the bimap $\circ_{G_i}$ is represented by 
$\left\{\left[\begin{smallmatrix} 0 & I_4 \\ -I_4 & 0\end{smallmatrix}\right],
\left[\begin{smallmatrix} 0_4 & L_i\\  -L_i^t &  0_4 \end{smallmatrix}\right]\right\}$,
where
\begin{align*}
	L_1 & =  \begin{bmatrix} 0 & & \\  & 1 & \\ & & 0 & 1 \\ & & b & a \end{bmatrix} &
	L_2 & =  \begin{bmatrix} 0 & & \\  & 0 & \\ & & 0 & 1 \\ & & b & a \end{bmatrix}.
\end{align*}
Since $L_1$ has three distinct eigenvalues in $K$, and $L_2$ only two,
the centralizers, $C(L_1)$ and $C(L_2)$, are non-isomorphic algebras.  
For $i=1,2$,
by~\cite{BW:slope}*{Lemma 3.2}, $A(\circ_{G_i})$ is isomorphic to $\mathbb{M}_2(C(L_i))$,
so $A(\circ_{G_1})$ and $A(\circ_{G_2})$ 
are non-isomorphic algebras. It follows that 
$G_1$ and $G_2$ are non-isoclinic.
\end{ex}

\subsection{A characterization of the indecomposable groups of genus 2}
\label{subsec:indecomp-proof}
We are almost ready to prove Theorem~\ref{thm:indecomps}.
Our approach requires that we examine the adjoint ring of the
commutation bimap of these groups in greater depth.
In particular, we provide a rather complete description of the bimap
obtained by forming a tensor product over such rings (Theorem~\ref{thm:tensor-bimap}). 
As well as helping us 
prove Theorem~\ref{thm:indecomps}, this will provide the 
foundation for the adjoint-tensor isomorphism test for the centrally
indecomposable groups of genus 2 that we present in Section~\ref{sec:iso-ind-genus2}.
\medskip

We will need the following convenient characterization of the adjoint ring of an indecomposable
bimap of the sloped type.

\begin{lemma}[\cite{BW:slope}*{Lemma 3.2(i)}]
\label{lem:centralizer}
Let $\circ\colon k^d\times k^d\bmto k^2$ be an alternating bimap
represented by a pair $\{\Phi_1,\Phi_2\}$ of forms with $\Phi_1$ invertible. Then
\begin{align*}
	A(\circ) & = C_{\mathbb{M}_d(k)}(\sigma), & \mbox{where}~\sigma  = \Phi_2\Phi_1^{-1}.
\end{align*}
In particular, $Z(A(\circ))=k[\sigma]\cong k[x]/(m(x))$, where $m(x)$ is the minimum polynomial of 
$\sigma$. If $\circ$ is indecomposable, then $m(x)=a(x)^e$ with $a(x)$ irreducible.
\end{lemma}

Note, Lemma~\ref{lem:centralizer} 
requires only that $\Phi_1$ is invertible and
makes no assumption about the indecomposability of the bimap --
we shall have more to say on this point after we prove the following 
crucial result.

\begin{thm}
\label{thm:tensor-bimap}
Let $k$ be a field and $\circ\colon k^{2n}\times k^{2n}\bmto k^2$ 
an indecomposable, alternating bimap represented by a pair $\{\Phi_1,\Phi_2\}$ 
with $\Phi_1$ invertible. Let
$\sigma=\Phi_2\Phi_1^{-1}$, and let $m(x)\in k[x]$ be
the minimal polynomial of $\sigma$. Then 
\begin{align}
\label{eq:wedge-iso}
	k^{2m}\otimes_{A(\circ)} k^{2n} & = k^{2n}\wedge_{A(\circ)} k^{2n} \cong k[x]/(m(x)),
\end{align} 
and $\otimes\colon k^{2n}\times k^{2n}\bmto k^{2n}\otimes_{A(\circ)} k^{2n}$ 
is a fully nondegenerate alternating $k[x]/(m(x))$-form.
Furthermore, the isomorphism in~(\ref{eq:wedge-iso}) can be computed in polynomial time.
\end{thm}
\begin{proof}
By Lemma~\ref{lem:isotropics}, there is a decomposition
$k^{2n}=E\oplus F$ with $E\circ E=0=F\circ F$, and by Theorem~\ref{thm:transitive}
this decomposition is unique up to a choice of basis. Thus, we may assume
\begin{align}
\label{eq:slope-pair}
\Phi_1=
\begin{bmatrix}
0 & I_n  \\
-I_n & 0  \\ 
\end{bmatrix}
~~~\mbox{and}~~~
\Phi_2=
\begin{bmatrix}
0  & \Psi \\
-\Psi^{\tr} & 0
\end{bmatrix},
\end{align}
where $\Psi$ is in Rational Canonical Form, so that
\begin{align}
\label{eq:slope}
	\sigma=\Phi_2\Phi_1^{-1} & = \begin{bmatrix}
		\Psi & 0 \\ 0 & \Psi^{\tr}
	\end{bmatrix}.
\end{align}
By Lemma~\ref{lem:centralizer}, $A(\circ)=C_{\mathbb{M}_{2n}(k)}(\sigma)=
\mathbb{M}_2(C_{\mathbb{M}_n(k)}(\Psi))$.  
The structure of centralizer matrices is well-studied and is determined by the representation of $\Psi$.  
In particular, there is a divisor chain $a_s(x)|a_{s-1}(x)|\cdots |a_1(x)$ of
$m(x)=a_1(x)$ and, using companion matrices $C(a_i)$, we have
\begin{align}
	\Psi = \diag(C(a_1), \ldots, C(a_s)).
\end{align}
Correspondingly $E=E_1\oplus \cdots \oplus E_s$ as a $k[x]$-module, with
$x$ acting as $\Psi$. 
The centralizer of $\Psi$ is a checkered matrix,
\begin{align*}
	C_{\mathbb{M}_{n}(k)}(\Psi) & = \left\{ [[M_{ij}]] : 1\leq i,j\leq m, ~~M_{ij}\in \hom_{k[x]}(E_i,E_j)\right\}.
\end{align*}
Now, the essential trick is to observe that for $\alpha_i \in E_i$, 
because $E_1$ is a faithful 
representation of $k[\Psi]$, 
$\hom_{k[x]}(E_1,E_i)=\hom_{k[x]}(k[x],E_i)\cong E_i$ as 
$k[x]$-modules.
Therefore, there exists a matrix
\begin{align*}
M_{\alpha} &=
\begin{bmatrix}
\tilde{\alpha}_1 & \tilde{\alpha}_2 & \cdots & \tilde{\alpha}_s \\
0 & 0 & \cdots & 0 \\
\vdots & \vdots & \ddots & \vdots \\
0 & 0 & \cdots & 0
\end{bmatrix}\in C_{\mathbb{M}_{n}(k)}(\Psi),
& e_1M_{\alpha} & =(\alpha_1,\ldots,\alpha_s)=\alpha\in E.
\end{align*}
Thus, $E=e_1 C_{\mathbb{M}_{d/2}(k)}(\Psi)$ is a cyclic $k[x]$-module, as is $F$.

We can now elucidate the structure of the bimap $\otimes_A$. Let $\alpha=(\alpha_1,\ldots,\alpha_s)\in E$ and 
$\beta=(\beta_1,\ldots,\beta_s)\in F$. Then,
\begin{align*}
\alpha \otimes_A \beta &=
e_1\begin{bmatrix} M_{\alpha} &  \\  & 0_n \end{bmatrix} \otimes_A
f_1
\begin{bmatrix}
0_n &  \\
 & M_{\beta}
\end{bmatrix}\\
&=
e_1\begin{bmatrix} M_{\alpha} &  \\  & 0_0 \end{bmatrix}\begin{bmatrix} M_{\beta}^{{\rm tr}} &  
\\  & 0_n \end{bmatrix} \otimes_A f_1\\
&= (\alpha \cdot \beta)(e_1\otimes_Af_1).
\end{align*}
here $\alpha\cdot\beta=\alpha_1\beta_1+\cdots  +\alpha_m\beta_m$ is the usual dot-product.
In particular, $k^{2n}\otimes_A k^{2n}=k[x](e_1\otimes f_1)$ is a cyclic 
$k[x]$-module, so 
the tensor product is a form.
All of the necessary constructions are carried out in polynomial time so the result follows.
\end{proof}

\begin{remark}
\label{rem:general-slope}
Although our application to Theorem~\ref{thm:indecomps}
concerns {\em indecomposable} bimaps of genus 2,
Theorem~\ref{thm:tensor-bimap} again requires only that $\Phi_1$ is invertible 
(just like Lemma~\ref{lem:centralizer}). 
This is explained in Section~\ref{sec:iso-genus2}.
We extend the notion of ``sloped" to any alternating bimap of genus 2
represented by $\{\Phi_1,\Phi_2\}$ with $\Phi_1$ invertible, 
and refer to $\sigma=\Phi_2\Phi_1^{-1}$ as a {\em slope}
of $\circ$. The slope is crucial to the work in~\cite{BW:slope} but also features
in earlier works such as~\citelist{\cite{GG}\cite{B-F}}.
\end{remark}

We are finally ready to prove Theorem~\ref{thm:indecomps}, restated here
for convenience.

\begin{thm*}
Let $G$ be a centrally indecomposable $p$-group of genus $2$ over a field $k$. 
Then $G$ is isoclinic to
one of the following two types of groups:
\begin{enumerate}[(i)]
\item (sloped case) a central quotient of a Heisenberg group,
\begin{align*}
	H &  =\left\{ \begin{bmatrix}
	1 & e & w \\ 0 & 1 & f \\ 0 & 0 & 1
	\end{bmatrix} \colon 
	\begin{array}{c} e,f,w\in k[x]/(a(x)^c),\\ a(x) \textnormal{ irreducible } \end{array}\right\},
\end{align*} by a subgroup $L\leq [H,H]\cong k^{m}$, $m=c\deg m(x)$, which is a $k$-subspace of codimension $2$; or

\item (flat case) the matrix group
\begin{align*}
	H^{\flat} & = 
	\left\{
	\left[\begin{array}{c|c|c}
		I_2 & \begin{array}{cccc} 
			e_1 & \cdots & e_m & 0 \\
			0 & e_1 & \cdots & e_m
		\end{array} & 
		\begin{array}{c} w_1 \\ w_2 \end{array}\\
		\hline		
		 & I_{m+1} & \begin{array}{c} f_0\\ \vdots \\ f_{m}\end{array}\\
		 \hline
		 & & 1
	\end{array}\right] \colon e_i, f_j,w_{\ell} \in k\right\}.
\end{align*} 
\end{enumerate}
\end{thm*}

\noindent {\em Proof.}
Following Proposition~\ref{prop:slope-flat} we know that every centrally indecomposable group 
$G$ of genus $2$ determines a pair of alternating forms $\{\Psi_1,\Psi_2\}$ of the form 
of (\ref{eq:slope-pair}). It remains to connect the two possible matrix pairs to the corresponding 
matrix groups state in Theorem~\ref{thm:indecomps}.

Suppose $G$ is sloped and let $m(x)=a(x)^c$ be the minimum polynomial of 
$\circ_G:k^{2n}\times k^{2n}\bmto k^2$. Set $H=H(R)$ -- the Heisenberg group over $R=k[x]/(a(x)^c)$.  
Then  the commutation bimap $\circ_H\colon R^{2}\times R^{2}\bmto R$ is an alternating {\em $R$-form}.  Choose an isomorphism $\phi:H/Z(H)\to G/Z(G)$
(both are isomorphic to $k^{2n}$).
By Theorem~\ref{thm:tensor-bimap}, $\circ_G$ factors through $\circ_H$ yielding 
a projection $\pi\colon R\to G'$
with $(u\phi\circ_H v\phi)\pi= u\circ_G v$.  This gives rise to an isomorphism $\hat{\phi}\colon H'/\ker\pi\to G'$,  
and $(\phi,\hat{\phi})$ is a pseudo-isometry from $\circ_G$ to $\circ_H$. Furthermore, if
\begin{align*}
	M & = \left\{\begin{bmatrix} 1 & 0 & w \\ 0 & 1 & 0 \\ 0 &0 & 1 \end{bmatrix} : w\pi=0\right\},
\end{align*}
then $G$ is isoclinic to $H/M$.

Next, we consider the flat case.  Let $E_{ij}$ indicate the matrix with $1$ in position $ij$ and $0$ elsewhere. Set
\begin{align*}
	M_a & = I+E_{1(a+2)}+E_{2(a+3)}, & N_b & = I+E_{(2+b)(n+4)},~\mbox{and} & Z_c & = I+E_{c(n+4)}.
\end{align*}
Then $H^{\flat}=\langle M_1,\dots,M_n,N_0,\dots,N_{n},Z_1,Z_2\rangle$ and $[H^{\flat},H^{\flat}]=\langle Z_1, Z_2\rangle$.
From matrix multiplication with respect to these generators we find the system of form derived from $H^{\flat}$ is the unique
 indecomposable flat pair of dimension $2n+1$.  
 Thus, if $G$ is centrally indecomposable and flat, then $G$ and $H^{\flat}$ are isoclinic.
 \hfill $\Box$

\subsection{Generalized discriminants and Pfaffians}
\label{subsec:pfaffian}
We have previously stated that Theorem~\ref{thm:indecomps} by itself is not enough
to decide isomorphism among groups of genus 2 over a field $k$. So, what else is needed?
The main result of this section provides a necessary and sufficient condition for isomorphism
between groups of genus 2 whose indecomposable central factors are all sloped.
This in turn gives rise to an isomorphism test that is effective when $|k|$ is small.
\medskip

In the foregoing discussion of the bimap $\circ=\circ_G\colon V\times V\bmto W$ 
associated to such a group, $G$, we have worked exclusively with 
the $k$-space $V=G/Z(G)$.  Now we turn our attention to $W=G'=k^2$.
Recall from Proposition~\ref{prop:autoclinism} that
\[
1\to C_{\Aut(G)}(Z(G))\to \Aut(G)\to \Aut(W)
\]
is an exact sequence.
As $\Aut(G)$ acts on the centroid, $C(\circ)=k$,
its action on $W$ is $k$-semilinear so we may replace $\Aut(W)$ with ${\rm \Gamma L}(2,k)$.
If $\{\Phi_1,\Phi_2\}$ is a pair of alternating forms representing $\circ$, we wish
to study the action of the pseudo-isometry group $\pseudo(\circ)$ on the 
2-dimensional $k$-space spanned by this pair. 
In particular, we are interested in deciding when (and how) an
element of ${\rm \Gamma L}(2,k)$ lifts to $\pseudo(\circ)$.
We begin by generalizing the notions of ``discriminant" 
to arbitrary lists of square matrices, and of ``Pfaffian" to
pairs of alternating forms.
\medskip

The discriminant of a bilinear form $\Psi$ is the square class of the determinant.
In this way it is invariant up to isometry, since $\det(X\Psi X^{\tr})=\det(X)^2 \det(\Psi)$.
For systems $\{\Psi_1,\ldots,\Psi_m\}$ of  forms we define
the {\em generalized discriminant} as follows:
\begin{align*}
	\disc(\Psi_1,\ldots,\Psi_m)=\det(x_1 \Psi_1+\cdots + x_m\Psi_m)\in k[x_1,\dots,x_m].
\end{align*}
We shall work with such systems up to isotopism, 
which means we can modify by independent matrices $X$ and $Y$
to arrive at 
\begin{align*}
\disc(X\Psi_1 Y,\ldots,X\Psi_mY)=
\det(X)\det(Y)\disc(\Psi_1,\ldots,\Psi_m).
\end{align*}  
Thus, $\disc(\Psi_1,\ldots,\Psi_m)$ is a homogenous polynomial defined only up
to a non-zero scalar multiple -- this is an interesting isotopism invariant so long as 
$m>1$.  
\smallskip

Next, consider a pair $\{\Phi_1,\Phi_2\}$ of alternating forms
representing a sloped bimap of genus 2.  
There exist subspaces 
$E$ and $F$ of equal dimension relative to which,  
\begin{align}
\label{eq:hyp}
(i=1,2) & & \Phi_i= \begin{bmatrix} 0 & \Psi_i \\ -\Psi_i^t & 0 \end{bmatrix},
\end{align}
so $\disc(\Phi_1,\Phi_2)=\disc(\Psi_1,\Psi_2)^2$. We therefore define the {\em generalized Pfaffian}, 
\begin{align}
	\Pf(\Phi_1,\Phi_2)=\disc(\Psi_1,\Psi_2).
\end{align}
We will use a natural action of $\GammaL(2,k)$ on the
 homogeneous polynomials in $k[x,y]$.
 For $\hat{\alpha}=\left(\left[ \begin{smallmatrix} a & b \\ c & d \end{smallmatrix}\right],\tau\right)\in \GL(2,k) \rtimes \Gal(k)$,
 define
\begin{align}
\label{eq:homog-action}
 f^{\hat{\alpha}}(x,y) = f^\tau(ax+by,cx+dy).
\end{align}

We now integrate the Pfaffian of a sloped pair with our understanding 
of centrally indecomposable groups of genus 2 to interpret
an isomorphism invariant introduced by 
Vishnevetski{\u\i} \citelist{\cite{Vish:1}\cite{Vish:2}}
in the case when $k=\mathbb{Z}_p$.
A version of the following theorem was announced in~\cite{Vish:1}
but the proof only considered the forward direction. We need the converse 
for our  isomorphism test so we provide a complete proof.

\begin{thm}
\label{thm:det-method}
Let $\{\Phi_1,\Phi_2\}$ and $\{\Lambda_1,\Lambda_2\}$ be alternating $k$-forms,
each written relative to any fully refined orthogonal decomposition, so for $i=1,2$,
\begin{align*}
	\Phi_i & = \diag\left(\Phi_i^{(1)},\cdots, \Phi_i^{(s)}\right) &
	\Lambda_i & = \diag\left(\Lambda_i^{(1)},\cdots, \Lambda_i^{(t)}\right).
\end{align*}
For $\hat{\alpha}\in \GammaL(2,k)$, there is a pseudo-isometry $(\alpha,\hat{\alpha})$ from $\{\Phi_1,\Phi_2\}$ 
to $\{\Lambda_1,\Lambda_2\}$ if, and only if,
$s=t$ and there is a permutation $\sigma$ of $\{1,\ldots,s\}$ such that for all $i$,
\begin{align}
\label{eqn:det-method}
	\Pf\left(\Phi_1^{(i)},\Phi_2^{(i)}\right)^{\hat{\alpha}} & \equiv \Pf\left(\Lambda_1^{(i\sigma)},\Lambda_2^{(i\sigma)}\right) \pmod{k^{\times}}.
\end{align}
\end{thm}

\begin{proof}
We begin with the forward direction.  
Assume $(\alpha,\hat{\alpha})\in\GammaL(2n,k)\times \GammaL(2,k)$ 
is a pseudo-isometry from $\circ$ to $\bullet$.   
The transitivity result in Theorem~\ref{thm:transitive} 
may be recast in the language of orthogonal decompositions for the
associated bimaps. In particular, there is an
isometry -- the image of $C_{\Aut(G)}(G')$ -- carrying the basis of the fully refined orthogonal
decomposition of $\{\alpha\Phi_1\alpha^{{\rm tr}},\alpha\Phi_2\alpha^{{\rm tr}}\}$
to that of $\{\Lambda_1,\Lambda_2\}$. As the former has $s$ terms, and the latter $t$ terms, it follows that
$s=t$. Furthermore, there is a permutation $\sigma$ of $\{1,\ldots,s\}$ such that for all $i\in \{1,\dots,s\}$ and
$j\in \{1,2\}$,  $\left(\Phi_j^{(i)}\right)^{\hat{\alpha}}\equiv \Lambda_j^{(i\sigma)}$.  Finally, observe that 
$\Pf\left(\Phi_1^{(i)},\Phi_2^{(i)}\right)^{\hat{\alpha}}\equiv \Pf\left(\left(\Phi_1^{(i)}\right)^{\hat{\alpha}},\left(\Phi_2^{(i)}\right)^{\hat{\alpha}}\right)$ 
so the claim follows.  (All equivalences are modulo $k^{\times}$).
\medskip

Now we consider the converse.  
Let $\hat{\alpha}=(\hat{\mu},\tau)\in\GL(2,k)\rtimes\Gal(k)$ satisfy (\ref{eqn:det-method}).
It suffices to find, for each $i$, an independent $\alpha_i$ such that 
$(\alpha_i,\hat{\alpha})$ is a pseudo-isometry from $\left\{\Phi_1^{(i)},\Phi_2^{(i)}\right\}$ 
to $\left\{\Lambda_1^{(i\sigma)},\Lambda_2^{(i\sigma)}\right\}$.  In particular 
we can assume $\{\Phi_1,\Phi_2\}$ and $\{\Lambda_1,\Lambda_2\}$ 
are indecomposable.  Following Proposition~\ref{prop:slope-flat} 
we may assume that there are bases relative to which
\begin{align*}
\Pf(\Phi_1,\Phi_2)= \det(xI+yC) & & \Pf(\Lambda_1,\Lambda_2)=\det(xI+yD).
\end{align*}
We produce a suitable lift of $\hat{\alpha}$ 
using an LUP-decomposition of $\hat{\mu}$.
\smallskip

First, suppose that $\hat{\mu}$ fixes an indeterminant  of $k[x,y]$.  Then, either 
$ \Pf(\Phi_1,\Phi_2)^{\hat{\mu}}\equiv \det( xI + yM )$ for some $M$,  or  
$\Pf(\Phi_1,\Phi_2)^{\hat{\mu}} \equiv \det(xN + yC)$ for some $N$.
Note that either option is equivalent to $\Pf(\Lambda_1,\Lambda_2)=\det(Ix+Dy)$ 
since $\Pf(\Phi_1,\Phi_2)^{\hat{\mu}}\equiv \Pf(\Lambda_1,\Lambda_2)$.
As $\{\Phi_1,\Phi_2\}$ and 
$\{\Lambda_1,\Lambda_2\}$ represent
 indecomposable bimaps, so $\{I,M\}$ (or $\{N,C\}$ as the case may be)
 and $\{I,D\}$
 are indecomposable pairs of matrices.
 Hence, by the Kronecker-Dieudonn\'e theorem, there are matrices $X$ and $Y$ such that 
$X\{I,D\}Y=\{I,C\}$.  If $\mu=\diag(X,Y^{\tr})$, then $(\mu,\hat{\mu})$ is
a pseudo-isometry from $\{\Phi_1,\Phi_2\}$ to $\{\Lambda_1,\Lambda_2\}$,
so any lower or upper triangular $\hat{\mu}$ can be lifted.

Next, suppose $\hat{\mu}$ interchanges $x$ and $y$.   Thus, 
\begin{align*}
	\det(Ix+Dy) \equiv \Pf(\Lambda_1,\Lambda_2) \equiv \Pf(\Phi_1,\Phi_2)^{\hat{\mu}} \equiv \det( xC + yI ).
\end{align*}
Arguing as before, there are matrices $X$ and $Y$ such that 
$X\{I,D\}Y=\{C,I\}$ and if $\mu=\diag(X,Y^{\tr})$, then $(\mu,\hat{\mu})$ is
a pseudo-isometry from $\{\Phi_1,\Phi_2\}$ to $\{\Lambda_1,\Lambda_2\}$.

In the general case, $\hat{\mu}$ is the product of 
$\hat{\beta}\hat{\gamma}\hat{\delta}$  where $\hat{\beta}$ and $\hat{\gamma}$ 
fix $x$ or $y$, and $\hat{\delta}$ transposes or fixes them
(an LUP-decomposition). Here, we lift $\hat{\mu}$ with three iterations using
the two special cases already treated.


Let $\alpha=(\mu,\tau)\in\GL(2,k)\rtimes\Gal(k)$. Since both $\circ$ and $\bullet$ are $k$-bilinear, it follows that 
for all $u,v\in V$, 
\[ 
u^\tau\circ v^\tau = (u\circ v)^\tau.
\]
Therefore, with $\alpha=(\mu,\tau)$, the pair $(\alpha,\hat{\alpha})$ is a pseudo-isometry from $\circ$ to $\bullet$, and so the theorem follows.
\end{proof}



\section{The Adjoint-Tensor Method}
\label{sec:adjten}
Having developed the necessary foundation 
we turn now to our isomorphism tests. 
To emphasize that our
algorithms apply only to finite groups and fields we 
shall henceforth with $\mathbb{F}_q$ in place of $k$,
where $\Bbb{F}_q$ is an extension of $\Bbb{Z}_p$.
Via Proposition~\ref{prop:autoclinism}
questions of isomorphism between finite $p$-groups of class 2 are reduced to ones of pseudo-isometry
between $\mathbb{F}_q$-bimaps. 
Details of this reduction -- and the isomorphism tests
it leads to -- are given in Section~\ref{sec:iso-auto}. Our focus for
the time being is the following problem.
\bigskip

\begin{minipage}{0.9\textwidth}
\noindent {\sc PseudoIsometry}$\;(~\circ~,~\bullet~)$
\begin{description}
\item[Given] alternating $\mathbb{F}_q$-bimaps $\circ,\bullet\colon\mathbb{F}_q^d\times\mathbb{F}_q^d\bmto \mathbb{F}_q^e$
\item[Return]  a pseudo-isometry from $\circ$ to $\bullet$, if such exists.
\end{description}
\end{minipage}
\bigskip

We can, in principle, return {\em all} pseudo-isometries
from $\circ$ to $\bullet$ as a coset of the group $\pseudo(\circ)$. That group is often the focus
of attention, in fact, because it relates directly to the automorphism group of a $p$-group.
We shall concentrate here on testing for pseudo-isometry and explain how to adapt our methods
to finding generators for $\pseudo(\circ)$ in Section~\ref{sec:pseudo-group}.
\medskip

The question of testing alternating bimaps for pseudo-isometry is one that arises also
in the generic method for group isomorphism -- the {\em nilpotent quotient algorithm} -- 
though framed in cosmetically different terms. The basic approach is as follows.
As both bimaps are alternating, they factor through the alternating tensor 
bimap $\wedge\colon\mathbb{F}_q^d\times\mathbb{F}_q^d\bmto  \mathbb{F}_q^d\wedge \mathbb{F}_q^d$
with induced maps $\hat{\circ},\hat{\bullet}\colon \mathbb{F}_q^d\wedge \mathbb{F}_q^d\to\mathbb{F}_q^e$.
The general linear group, $\GL(d,\mathbb{F}_q)$, acts naturally on $\mathbb{F}_q^d\wedge\mathbb{F}_q^d$,
and $\circ$ and $\bullet$ are pseudo-isometric if, and only if, an element of $\GL(d,\mathbb{F}_q)$
maps $\ker\hat{\circ}$ to $\ker\hat{\bullet}$. 
Thus, to determine pseudo-isometry, we must solve a {\em subspace transporter problem},
which is notoriously difficult even for ``well understood" actions like the exterior square representation.
In practice, it is possible to proceed by a direct orbit calculation only
for quite modest values of $d$ and $q$.

On the other hand, if $\circ$ and $\bullet$ have a constrained structure -- 
such as the ones arising in~\cite{LW} -- specialized techniques may be developed to compute orbits
efficiently. Inspired by the need to bridge the gap between slow, generic methods,
and very fast, highly specialized ones, in~\cite{BW:autotopism} the first and third authors
proposed a new general technique called the {\em adjoint-tensor method}.
The method, which we outline in general below, 
is particularly well-suited
to the alternating bimaps of genus 2. 
Most of remaining content of the paper is concerned with the particular
application of adjoint-tensor to that case.
\medskip

\begin{figure}[!htbp]
\noindent {\sc PseudoIsometry}$\;(~\circ,\bullet\colon \mathbb{F}_q^d\times\mathbb{F}_q^d\bmto\mathbb{F}_q^e~)$\hfill
\smallskip

\begin{tabular}{ll}
1. & Compute $A=A(\circ)$ and $A(\bullet)$. \\

2. & Test if there exists $\rho\in\Gamma{\rm L}(d,\mathbb{F}_q)$ with $A(\bullet)^{\rho}=A$; if not, return {\tt false}. \\

 & Replace $\bullet$ with an isometric bimap, $\star$, so that $A=A(\circ)=A(\star)$. \\

3. & Compute kernels for the induced maps $\hat{\circ},\hat{\star} \colon \mathbb{F}_q^d\wedge\mathbb{F}_q^d\to\mathbb{F}_q^2$. \\

4. & Construct generators for $\pseudo(\wedge_A)$. \\

5. & Find $(\phi,\hat{\phi})\in\pseudo(\wedge_A)$ with $(\ker\hat{\circ})\hat{\phi}=\ker\hat{\star}$; \\
     & compute and return the pseudo-isometry $(\rho\phi,\hat{\phi})$ from $\circ\to\bullet$.
\end{tabular}
\caption{The adjoint-tensor approach to solving {\sc PseudoIsometry}.}\label{fig:code}
\end{figure}

In step 1 we are computing the adjoint rings.  
This is no worse than solving a system of $ed^2$ equations in $2d^2$ variables.  
In certain situations -- notably
sloped genus 2 bimaps -- one can extend the practical range by avoiding these
large linear systems~\cite{BW:slope}.

Step $2$ does not in general have a known polynomial-time solution.  
It asks whether subalgebras of $\mathbb{M}_d(\mathbb{F}_q)$ are conjugate
which, beyond just being isomorphic, requires that they are
identically represented on $\mathbb{F}_q^d$. This leads to
the notion of {\em module similarity}, a problem which was shown
in~\cite{BW:mod-iso} to be as hard as graph isomorphism.
If we are able to find such a $\rho$, however, we replace
$\bullet$ with the bimap $u\star v = u\rho\bullet v\rho$,
so that $A(\star)=A(\circ)=A$.  

To understand step $3$, recall from Section~\ref{sec:decomps-genus2} 
that the adjoint ring $A=A(\circ)$ 
is the largest ring, $B$, such that $\circ$ factors through the tensor product
$\mathbb{F}_q^d\otimes_{B}\mathbb{F}_q^d$. 
Since, for us, the bimap $\circ$ is alternating, it additionally factors
through the exterior product 
$\mathbb{F}_q^d\wedge_A\mathbb{F}_q^d$ (cf.~Theorem~\ref{thm:tensor-bimap}),
so there is an induced map $\hat{\circ}\colon \mathbb{F}_q^d\wedge_A\mathbb{F}_q^d\to\mathbb{F}_q^e$
such that
\begin{align}
\label{eq:induce}
(\forall u,v\in V) & & u\circ v=(u\wedge v)^{\hat{\circ}}.
\end{align}
Computing $\ker \hat{\circ}$ and $\ker \hat{\star}$ 
amounts to solving a system of $O(d^2)$ linear equations.

Step 4 builds the group that acts on $V\wedge_A V$.  As we noted earlier, $\GL(V)$ acts naturally
on the components of the traditional exterior square $V\wedge V$.  
To respect the tensor over $A$ we must instead use the group $\pseudo(\wedge_A)$,
the structure of which is described in \cite{BW:autotopism}*{Theorem~4.5}.
The description requires one to compute the normalizer of $A$,
and the complexity of this problem depends critically on structural properties
of $A$ and on its representation.

The final component (step 5) is the same as the conclusion of the nilpotent quotient algorithm described earlier. 
Once again $\circ$ and $\star$ are pseudo-isometric if,
and only if, $\ker\hat{\circ}$ and $\ker\hat{\star}$ are in the same orbit, this time 
under the action of $\pseudo(\wedge_A)$. Hence, we must solve the subspace
transporter problem for the representation of $\pseudo(\wedge_A)$
on $\mathbb{F}_q^d\wedge_A\mathbb{F}_q^d$.
\smallskip

In summary,
steps 2 (module similarity), 4 (normalizers of matrix rings), and 5 (subspace transporters) are each known 
to be at least as hard as graph isomorphism 
\citelist{\cite{BW:mod-iso}*{Theorem~1.2}\cite{BW:autotopism}*{Theorem~4.5(iii)}\cite{LM:normalizer}}. It seems, then, that
adjoint-tensor merely turns one difficult problem into three!  
The idea, though, is that each new ``hard problem" is either smaller in size, 
or has a controlled structure that admits a more efficient solution.  
{\em This is exactly the case for groups of genus $2$.}



\section{Indecomposable Bimaps of Genus 2}
\label{sec:iso-ind-genus2}
We now restrict to bimaps of genus 2 and develop an effective algorithm
for {\sc PseudoIsometry} in this case. We start in this section
by further restricting to (orthogonally) indecomposable bimaps.
Our goal is the following result.

\begin{thm}
\label{thm:indecomp-bimap-alg}
There is a polynomial-time algorithm that, given 
indecomposable, alternating bimaps
$\circ,\bullet\colon \mathbb{F}_q^d\times\mathbb{F}_q^d\bmto
\mathbb{F}_q^2$
of genus 2, decides if the bimaps are pseudo-isometric and, if so, 
constructs 
a pseudo-isometry, namely $(\phi,\hat{\phi})\in \Gamma{\rm L}(d,\mathbb{F}_q)\times\Gamma{\rm L}(2,\mathbb{F}_q)$ 
such that $u\phi\bullet v\phi= (u\circ v)\hat{\phi}$ for all $u,v\in\mathbb{F}_q^d$.
The algorithm is deterministic if $p$ is bounded and Las Vegas otherwise.
\end{thm}

The qualification of Las Vegas versus deterministic in Theorem~\ref{thm:indecomp-bimap-alg}
(and in later theorems) arises only from the
need to factor polynomials.

\subsection{Standard indecomposable pairs of matrices}
\label{subsec:standard}
To prove Theorem~\ref{thm:indecomp-bimap-alg} we apply the ``flat-sloped" dichotomy of
Theorem~\ref{thm:indecomps} to the associated pairs $\{\Phi_1,\Phi_2\}$ of alternating forms. 
It will be helpful to select a basis relative to which
\begin{align}
\label{eq:pair}
\Phi_1=\begin{bmatrix} 0 & \Psi_1 \\ -\Psi_1^{{\rm tr}} & 0 \end{bmatrix}~~\mbox{and}~~
\Phi_2 = \begin{bmatrix} 0 & \Psi_2 \\ -\Psi_2^{{\rm tr}} & 0 \end{bmatrix}, 
\end{align}
and $\{\Psi_1,\Psi_2\}$ is given by the appropriate part of Theorem~\ref{thm:Dieudonne}.
We begin by finding a totally isotropic decomposition for the pair (see Lemma~\ref{lem:isotropics}). 
This is done in polynomial time by finding a hyperbolic pair of idempotents in the adjoint ring of the pair,
as we did in the proof of Theorem~\ref{thm:cent-hyp-alg}(ii).
By changing to a basis that respects this totally isotropic decomposition, we obtain a pair of
forms as in (\ref{eq:pair}) with $\{\Psi_1,\Psi_2\}$ an arbitrary indecomposable pair of matrices.
It remains to find matrices $X,Y$ such that $\{X\Psi_1Y,X\Psi_2Y\}$ has the desired form.

The sloped case, namely Theorem~\ref{thm:Dieudonne}(i), has been discussed from
the point of view of algorithms in several recent papers; see~\citelist{\cite{GG}\cite{BW:slope}}
for example. The conversion depends only on the sloped aspect, and so works for decomposable pairs.

Suppose, then, that $\Psi_1,\Psi_2\in\mathbb{M}_{n,n+1}(\mathbb{F}_q)$ is an indecomposable pair, where $d=2n+1$.
Compute $X\in \mathbb{M}_{n}(\mathbb{F}_q), Y\in \mathbb{M}_{n+1}(\mathbb{F}_q)$ such that
$X\Psi_1Y=[I_n|0]$, the standard matrix for $\Psi_1$, using Gaussian elimination.  We now
modify $X$ and $Y$ so that $X\Psi_1Y=[I_n|0]$ and $X\Psi_2Y=[0|I_n]$.
We do this by successive approximations.

First, write $X\Psi_2Y=[U|u^{{\rm tr}}]$, and find $B\in\mathbb{GL}_n(\mathbb{F}_q)$ such that $BUB^{-1}=R$ is
in generalized Jordan normal form. 
Reassign $X:=BX$ and $Y:=Y\left[\begin{smallmatrix} B^{-1} &  0 \\ 0 & 1 \end{smallmatrix}\right]$.
As the pair is indecomposable, $R$ is a single companion matrix, 
say $R=\left[ \begin{smallmatrix} 0 & I_{n-1} \\ \alpha & v' \end{smallmatrix} \right]$.
Secondly, write $X\Psi_2Y=[R|v^{{\rm tr}}]$ and 
find $T$ in the cyclic algebra generated by $R$ sending $v^{{\rm tr}}$ to 
$(0\ldots 0 1)^{{\rm tr}}$. Reassign $X:=TX$.
Finally, write $X\Psi_2Y=\left[ \begin{smallmatrix} 0 & I_{n-1} & 0 \\ \beta & b' & 1 \end{smallmatrix} \right]$,
put $b:=(\beta~b')\in\mathbb{F}_q^n$, and reassign $Y:=Y\left[ \begin{smallmatrix} I_n & 0 \\ -b & 1 \end{smallmatrix} \right]$.

\subsection{The flat case}
\label{subsec:flat-indec}
It is immediate from our discussion of this case in the preceding section that
 two flat, indecomposable bimaps of genus 2 are isometric, and 
Theorem~\ref{thm:indecomp-bimap-alg} holds in this case.
Recall, however, we shall eventually require generators for $\pseudo(\circ)$.
We address this problem for general bimaps later in Section~\ref{sec:pseudo-group}, but we
can resolve the matter now for flat,
indecomposable bimaps of genus 2.

\begin{prop}
\label{prop:pseudo-flat}
If $\circ\colon \colon \mathbb{F}_q^d\times\mathbb{F}_q^d\bmto
\mathbb{F}_q^2$ is a flat, indecomposable bimap of genus 2,
then there is an epimorphism $\pseudo(\circ)\to\Gamma{\rm L}(2,\mathbb{F}_q)$
with kernel $\isom(\circ)$.
\end{prop}

\begin{proof}
For $e=(e_1,\ldots,e_n)\in\Bbb{F}_q^n$, define $M=M(e)=\left[\begin{smallmatrix}
	e_1 &  \cdots & e_n & 0 \\
	0 & e_1 & \cdots & e_n 
	\end{smallmatrix}\right]$
and set 
\begin{align*}
E &= \{M(e)\colon  e\in \mathbb{F}_q^n\}\leq \M_{2\times (n+1)}(\mathbb{F}_q).
\end{align*}
Then, the usual matrix multiplication 
\begin{align*}
\times \colon & E\times \mathbb{M}_{(n+1)\times 1}(\mathbb{F}_q)\bmto \mathbb{M}_{2\times 1}(\mathbb{F}_q)
\end{align*} 
is described by the system of forms
$\Psi=\{[I_n |0 ],[0|I_n]\}$.  As $\circ$ is given a pair of forms $\Phi_i = 
\left[\begin{smallmatrix} 0 & \Psi_i \\ -\Psi_i^{tr} & 0 \end{smallmatrix}\right]$,
it follows that every isotopism $(\alpha,\beta;\gamma)$ of $\times$ 
induces a pseudo-isometry $(\alpha\oplus \beta,\gamma)$ of $\circ$,
so it suffices to show that $\Gamma{\rm L}(2,\mathbb{F}_q)$ lifts to autotopisms of $\times$ 
acting faithfully on $\M_{2\times 1}(\Bbb{F}_q)$.

For $e=(e_1,\ldots,e_n)\in\Bbb{F}_q^n$, 
define $f_{e}(x,y)=e_1 x^{n}+\cdots + e_i x^{n-i}y^i +\cdots + e_n y^n\in \Bbb{F}_q[x,y]$.  
Then $M_e\mapsto f_e(x,y)$ is a linear bijection from 
$E=\{M(e) \colon e\in \Bbb{F}_q^n\}$ to the set of homogeneous polynomials
in $\Bbb{F}_q[x,y]$ of degree $n$.  Let
$\rho\colon \Gamma{\rm L}(2,\mathbb{F}_q)\to \Gamma{\rm L}(n+1,\Bbb{F}_q)$
denote the faithful representation arising from the natural action of $\Gamma{\rm L}(2,\Bbb{F}_q)$
on the latter.  
For $g\in \Gamma{\rm L}(2,\mathbb{F}_q)$, define
\begin{align*}
	M(e)\lambda_g & := g M(e) (g^{-1}\rho).
\end{align*}
Then $(\lambda_g,g\rho;g)$ is an isotopism of $\bullet$, and the result follows.
\end{proof}

\subsection{The sloped case}
\label{subsec:sloped-indec}
Recall that an alternating bimap $\circ\colon \mathbb{F}_q^d\times\mathbb{F}_q^d\bmto \mathbb{F}_q^2$
of genus 2 is sloped if we can represent it by a pair $\{\Phi_1,\Phi_2\}$ with $\Phi_1$
nondegenerate. Our goal is to complete the proof of Theorem~\ref{thm:indecomp-bimap-alg}
by presenting a test for pseudo-isometry between two sloped, indecomposable
bimaps $\circ,\bullet\colon \mathbb{F}_q^d\times\mathbb{F}_q^d\bmto
\mathbb{F}_q^2$.
\medskip

We will use the adjoint-tensor method of Section~\ref{sec:adjten}, referring to the pseudo-code 
in Figure~\ref{fig:code}.  Recall that we must resolve three problems:
\begin{enumerate}
\item[\underline{Line 2}] Given adjoint algebras $A(\circ)$ and $A(\bullet)$, find $\rho\in\Gamma{\rm L}(d,\mathbb{F}_q)$
with $A(\bullet)^{\rho}=A(\circ)$ (if such exists).
\item[\underline{Line 4}] Given $A=A(\circ)$, build generators for $\pseudo(\wedge_A)$.
\item[\underline{Line 5}] Solve the ``transporter problem": 
given subspaces $U,V$ of $\mathbb{F}_q^d\otimes_A\mathbb{F}_q^d$
find $(\phi,\hat{\phi})\in \pseudo(\wedge_A)$
sending $U$ to $V$, or prove that no such $(\phi,\hat{\phi})$ exists.
\end{enumerate}
We consider each problem in turn.

\subsubsection{Conjugating the adjoint algebras}
\label{subsec:conj}
As we noted in Section~\ref{sec:adjten}, conjugacy of algebras is very hard in general,
but an efficient solution exists in our setting.
This relies on the special nature of adjoint algebras for sloped bimaps 
of genus $2$.  
Any such a bimap $\circ$ is represented by a pair $\{\Phi_1,\Phi_2\}$ of alternating
forms with $\Phi_1$ invertible, and slope $\sigma=\Phi_2\Phi_1^{-1}$ 
is invariant under basis change in $\mathbb{F}_q^d$  -- that is, invariant in $\isom(\circ)$.  By 
Lemma~\ref{lem:centralizer}, $A(\circ) = C_{\mathbb{M}_d(\mathbb{F}_q)}(\sigma)$, so
\[
Z(A(\circ))=\mathbb{F}_q[\sigma]\cong \mathbb{F}_q[x]/(m(x)),
\]
where $m(x)$ is the minimum polynomial of 
$\sigma$.  The conjugacy problem for cyclic algebras has an efficient solution.

\begin{thm}[\cite{BW:mod-iso}*{Theorem~1.3}]
\label{thm:conj-cyclic}
There is a polynomial-time algorithm that, given cyclic algebras $A=\mathbb{F}_q[\alpha]$, $B=\mathbb{F}_q[\beta]$,
for $\alpha,\beta\in\mathbb{M}_d(\mathbb{F}_q)$,
finds $\rho\in\Gamma{\rm L}(d,\mathbb{F}_q)$ with $A^{\rho}=B$, or decides that no such $\rho$ exists.
\end{thm}

This leads to a resolution of our first problem. 

\begin{thm}
\label{thm:conjugacy}
There is a polynomial-time algorithm that, given sloped bimaps 
$\circ,\bullet\colon \mathbb{F}_q^d\times\mathbb{F}_q^d\bmto \mathbb{F}_q^2$ of genus $2$, 
finds $\rho\in\Gamma{\rm L}(d,\mathbb{F}_q)$ such that $A(\circ)^{\rho}=A(\bullet)$,
or decides that no such $\rho$ exists.  
\end{thm}

\begin{proof}
By Lemma~\ref{lem:centralizer}, $A(\circ)$ and $A(\bullet)$ are centralizers of 
slopes $\sigma_{\circ}$ and $\sigma_{\bullet}$, respectively.  
Furthermore, $A(\circ)$ and $A(\bullet)$ are conjugate if, and only, if their
centers are conjugate. The result now follows from Theorem~\ref{thm:conj-cyclic}.
\end{proof}

\subsubsection{The properties of $\wedge_A$}
\label{subsec:tensor}
To describe $\pseudo(\wedge_A)$ we need the following result.

\begin{thm}[\cite{BW:autotopism}*{Theorem~1.5}]
If $\circ\colon V\times V\bmto W$ is an alternating bimap with adjoint ring $A=A(\circ)$, then
$\pseudo(\wedge_A)$ is faithfully represented on $V$ as  
\begin{align*}
 N^*(A) &= \{g\in\GL(d,\F)\colon A^g=A~\mbox{and}~(x^g)^*=(x^*)^g~\mbox{for all}~x\in A\}.
\end{align*}
\end{thm}

Using the general structure of $N^*(A)$ laid out in~\cite{BW:autotopism}*{Theorem 4.5} 
together with Theorem~\ref{thm:tensor-bimap}, we gain a very detailed understanding of 
$\pseudo(\wedge_A)$ when $\circ\colon \mathbb{F}_q^d\times \mathbb{F}_q^d\bmto \mathbb{F}_q^2$ is 
a sloped, indecomposable bimap of genus 2.
Put 
\begin{align}
\label{eq:def-K}
	K=\mathbb{F}_q[x]/(m(x)),
\end{align}
where $m(x)$, the minimal polynomial of the slope of $\circ$, is a power
of an irreducible polynomial. Hence,
$K\cong L[t]/(t^e)$ where $L/\mathbb{F}_q$ is an algebraic field extension,
 and
\begin{align}
\label{eq:pseudo-tensor}
1 \longrightarrow \isom(\wedge_A) \longrightarrow \pseudo(\wedge_A) \longrightarrow
 \Gamma {\rm L}(1,K) \longrightarrow 1
\end{align}
is a short exact sequence,
where $\isom(\wedge_A)$ is the kernel of the action of $\pseudo(\wedge_A)$ on 
$V\wedge_A V$.  The algorithms to find generators and further structure of isometry groups were 
given in \citelist{\cite{BW:isom}\cite{BW:slope}}.  Hence, the group we must understand is
\begin{align*}
	\pseudo(\wedge_A)/\isom(\wedge_A) & \cong\Gamma{\rm L}(1,K)=K^{\times}\rtimes {\rm Aut}(K).
\end{align*}
In doing so we shall obtain
an alternative factorization of this group that we will need for the remaining computations.
\medskip

The group $\Aut(K)=\Aut(L[t]/(t^e))$ satsifies 
\begin{align*}
	1\to \Sigma \to \Aut(K)\to {\rm \Gamma L}(1,L)\to 1,
\end{align*}
where $\Sigma=C_{\Aut(K)}((t)/(t^2))$ is a quotient of the notorious {\em Nottingham group}, 
a well-studied pro $p$-group~\cite{LGMc}*{Section 12.4}. 
In particular generators for ${\rm \Gamma L}(1,L)=L^{\times}\rtimes {\rm Gal}(L/\mathbb{Z}_p)$ are
known.
The group $\Sigma$ consists of {\em substitution automorphisms},
\begin{align*}
	\Lambda_{a(t)} \colon p(t)  \mapsto p(a(t)),
\end{align*}
where $a(t)=t+a_2 t^2+\cdots$, and
is generated by $\{\Lambda_{t+t^2},\Lambda_{t+t^3}\}$.

\subsubsection{Solving the transporter problem}
\label{subsec:transporter}
Our final concern is to solve for the transporter problem: 
{\em given subspaces $U,V$ of $\Bbb{F}_q^d\wedge_A\Bbb{F}_q^d$,
find $(\alpha,\hat{\alpha})\in \pseudo(\wedge_A)$
sending $U$ to $V$, or prove that no such $(\alpha,\hat{\alpha})$ exists.}

Our algorithm handles the Galois group of $L$ by trial and error.
We  work with the remaining part -- namely with $G:=  \Sigma K^{\times}$ -- in a 
more refined manner. To facilitate our computations, we choose generators for $G$
that produce a convenient factorization.
First, as a consequence of Wedderburn's principal theorem, $K^{\times}$ factorizes as 
$Q_1\rtimes G_1$, with $Q_1$ unipotent, and $G_1$
isomorphic to the multiplicative group of a field. Secondly, as we saw above, there is an
analogous factorization, $Q_2\rtimes G_2$, of $\Sigma$. Put $Q:=Q_1Q_2$, and $J:=J(K)$, the Jacobson radical of $K$.
The crucial properties of the factorization $QG_1G_2$ for our purpose are as follows:
\smallskip

(i) $Q$ is a unipotent group;

(ii) there are fields $L_1,L_2$ such that $G_1=L_1^{\times}$ and $G_2=L_2^{\times}$; and

(iii) $G_1$ acts faithfully on the $k$-space $K/J$, and $G_2$ acts faithfully on $J/J^2$.
\smallskip

Before proceeding further, we require two different ``transporter" algorithms that
will solve our problem in special cases. The proof of the following result
generalizes an earlier algorithm of L. R\'{o}nyai developed for the case of fields; see \cite{LW}*{Lemma~4.8}.

\begin{lem}
\label{lem:Ronyai}
Let $R$ be a $k$-algebra of matrices, $k$ a finite field. 
Given $k$-subspaces $X,Y$ of $R$ with $\dim_k X=\dim_k Y$, in polynomial time one can find $r\in R^{\times}$ with
$Xr=Y$, or decide that no such $r$ exists.
\end{lem}

\begin{proof}
Assume that $X$ and $Y$ have equal dimension $t$.
First, find a basis for the $k$-space $\mathcal{S}=\{a\in R\colon Xa\subseteq Y\}$ as follows.
Let $b_1,\ldots,b_n$ be a $k$-basis for $R$. Fix bases 
for $X$ and $Y$. Let $y_1,\ldots,y_t$ be the basis for $Y$,
and write $y_q=\sum_{p=1}^n\gamma_{pq}b_p$ for $1\leq q\leq t$.
Now, for each basis element $x=\sum_{i=1}^n\alpha_ib_i$ of $X$,
we want all scalars $z_1,\ldots,z_n,w_1,\ldots,w_t$ such that
\begin{align*}
\left(\sum_{i=1}^n \alpha_ib_i\right)\left(\sum_{j=1}^nz_jb_j\right)
=\sum_{i,j=1}^n\alpha_iz_jb_ib_j &= \sum_{p=1}^t\sum_{q=1}^n\gamma_{pq}w_pb_q.
\end{align*}
Writing each $b_ib_j$ as a linear combination of $b_1,\ldots,b_n$ (these are the
{\em structure constants} of $R$ relative to our chosen basis) , 
a basis for $\mathcal{S}$ is obtained as the solution of the resulting linear
system in the unknowns  $z_1,\ldots,z_n,w_1,\ldots,w_t$
by projecting onto the $z_i$ coordinates.

Evidently, if $\mathcal{S}=0$, no $r\in R$ exists with $Xr\subseteq Y$, so we may assume that $\mathcal{S}\neq 0$.
Now we must find an {\em invertible} element in $\mathcal{S}$, if such exists. We present
a deterministic method, but in practice such an element is found more efficiently by random search.

Compute $\mathcal{T}=\{b\in R\colon Yb\subseteq X\}$ as above (interchanging the roles of $X$ and $Y$).
Form the set $\mathcal{S}\mathcal{T}=\{st: s\in \mathcal{S},t\in\mathcal{T}\}\subset \End(X)$.
Then using the algorithm of \cite{BL:mod-iso}*{Theorem 2.4} we prove that there are no invertible
elements in $\mathcal{S}\mathcal{T}$, or we construct an invertible element
$z$ of the subring generated by $\mathcal{S}\mathcal{T}$ as a product $z=s_1t_1s_2t_2\cdots s_nt_n$,
$s_i\in \mathcal{S}$, $t_i\in \mathcal{T}$.  As $s_1$ is injective and
$\dim_k X=\dim_k Y$, $Xs_1$ is mapped injectively into $Y$. It follows that $Xs_1=Y$.
\end{proof}

\begin{remark}
\label{rem:dual-trick}
We intend to apply Lemma~\ref{lem:Ronyai} in the case when $X$ and $Y$ have codimension 2 in $R$.
By translating the problem to the dual space of $R$, we can solve the transporter problem instead for 
spaces of dimension 2. This reduces the complexity of computing the $k$-spaces 
$\mathcal{S}$ and $\mathcal{T}$ by a factor of $O(d)$.
\end{remark}

The following  algorithm is a small application of the deeper theorem of \cite{Luks:mat}*{Theorem 3.2(7)}.  
It is also known by many as the ``unipotent stabilizer algorithm'' (see also~\cite{Ruth2}).

\begin{lem}
\label{lem:unipotent}
Let $Q$ be a unipotent subgroup of $\GL_d(\Bbb{F}_q)$.
Given subspaces $X,Y$ of $\Bbb{F}_q^d$, 
in polynomial time one can find $u\in Q$ with
$Xu=Y$ if such $u$ exists.
\end{lem}

We can now complete the description of our algorithm. Recall that $U$ and $V$ are given 
$\mathbb{F}_q$-subspaces of $K$, $J$ is the Jacobson radical of $K$, 
and $G=L_1^{\times}L_2^{\times}Q$.
We wish to decide if there exists $g\in G$ such that $Ug=V$.
\smallskip

First, construct the representation of $L_1$ on $K/J$, and use Lemma~\ref{lem:Ronyai} to 
find $g_1\in L_1^{\times}$ such that $Ug_1\equiv V\;({\rm mod}\;J)$, if such exists.
Put $U_1=Ug_1$.
Next, construct the representation of $L_2$ on $J/J^2$, and use Lemma~\ref{lem:Ronyai} 
again to find $g_2\in L_2^{\times}$
such that $U_1Jg_2\equiv VJ\;({\rm mod}\;J^2)$, if such exists.
Put  $U_2=U_1g_2$.
Finally, use Lemma~\ref{lem:unipotent} to find $w\in Q$
with $U_2w=V$, if such exists. Return $g:=g_1g_2w$. 
Note, if we failed to construct any one of the elements $g_1,g_2,w$, then there is
no element $g\in G$ transporting $U$ to $V$. 

\subsection{Proof of Theorem~\ref{thm:indecomp-bimap-alg}.} The correctness
of the algorithms presented in Sections~\ref{subsec:flat-indec} and~\ref{subsec:sloped-indec} 
has already been established. It remains to analyze complexity.

The sloped case in Section~\ref{subsec:sloped-indec} requires more analysis, and we proceed
one subsection at a time. First, conjugating the algebra $A(\bullet)$ to $A(\circ)$ is done by Theorem~\ref{thm:conjugacy}.
Secondly, building the tensor product $\V\wedge_A\V$ and generators of $\pseudo(\wedge_A)$ is
done in polynomial time in Section~\ref{subsec:tensor}.  That leaves Section~\ref{subsec:transporter}, which requires
more care. For each $\gamma\in {\rm Gal}(L)$ we seek $g\in \Aut(K)$ with $(V\gamma)g=U$.
This uses two calls to Lemma~\ref{lem:Ronyai}, and one call to Lemma~\ref{lem:unipotent},
which are both polynomial time. The overall complexity is therefore polynomial, since
$|\Gamma|\leq\frac{d}{2}$. \hfill $\square$



\section{General Bimaps of Genus 2}
\label{sec:iso-genus2}
We now consider arbitrary alternating 
bimaps $\circ,\bullet\colon \mathbb{F}_q^d\times\mathbb{F}_q^d\bmto \mathbb{F}_q^2$
of genus 2.  
Much of the work has already been done in the indecomposable setting above,
but we must now combine the results of various indecomposables.  
Here, the theory becomes difficult.  As Example~\ref{subsec:special-ex} 
shows,  for instance, indecomposable factors may be glued together in different ways 
to produce bimaps that are not pseudo-isometric.  
In spite of these challenges we  prove the following
extension of Theorem~\ref{thm:indecomp-bimap-alg}.

\begin{thm}
\label{thm:bimap-alg}
There is an algorithm that, given  
alternating bimaps
$\circ,\bullet\colon \mathbb{F}_q^d\times\mathbb{F}_q^d\bmto
\mathbb{F}_q^2$
of genus 2, decides whether the bimaps are pseudo-isometric and, if so, 
constructs $(\phi,\hat{\phi})\in \GL(d,\mathbb{F}_q)\times\GL(2,\mathbb{F}_q)$ 
such that $u\phi\bullet v\phi= (u\circ v)\hat{\phi}$ for all $u,v\in\mathbb{F}_q^d$.
The algorithm is polynomial time if $q$ is bounded or
the number of pairwise pseudo-isometric indecomposable summands
of the input bimaps is bounded.  If $p$ is bounded the algorithms are deterministic,
otherwise they are Las Vegas. 
\end{thm}

Not surprisingly, 
the first step is to find a fully-refined orthogonal decomposition of
the input bimaps, using Theorem~\ref{thm:cent-hyp-alg}(i).
By Theorem~\ref{thm:transitive}(i) the multiset of terms in such a
decomposition is unique up to pseudo-isometry.
Hence, if the terms in the two decompositions cannot be paired
up pseudo-isometrically, then the bimaps themselves are not pseudo-isometric.
In particular, if the multisets
of dimensions of indecomposables are different for the two bimaps,
then they are not pseudo-isometric. Furthermore, assuming the dimensions
of the flat indecomposables are compatible, $\circ$ and $\bullet$ are pseudo-isometric if,
and only if, their restrictions to the sum of the sloped parts are pseudo-isometric. 
Hence, we may assume that the indecomposable
factors of each bimap are sloped.

We reiterate that deciding pseudo-isometry of $\circ$ and $\bullet$ is not as straight-forward 
as matching up isomorphic sloped indecomposable factors -- more
subtlety is required. We present two rather different approaches. The first
is very effective when $|\mathbb{F}_q|$ is small, and is based directly on the theory
developed in Section~\ref{sec:genus2}. The second, which we use
for larger fields, is the adjoint-tensor method. 
Before proceeding we must first address a curiosity that 
can arise in our new setting. 

\subsection{A rare configuration}
\label{sec:small-field}
We are assuming that $\circ,\bullet\colon \mathbb{F}_q^d\times\mathbb{F}_q^d\bmto\mathbb{F}_q^2$
are nondegenerate bimaps whose indecomposable summands are sloped. 
We wish to assume that $\circ$ and $\bullet$ are sloped 
{\em globally}, meaning that we can choose
a pair $\{\Phi_1,\Phi_2\}$ of forms representing each one with $\Phi_1$ nondegenerate.
Clearly an initial choice $\{\Phi_1,\Phi_2\}$ can be made for which both forms are
degenerate. For example,
\begin{align*}
	\Phi_1 & = \left[\begin{array}{cc|cc}	& & 1 & \\ & & & 0\\ \hline -1 & & & \\ & 0 & & 	\end{array}\right]
&	\Phi_2 & = \left[\begin{array}{cc|cc}	& & 0 & \\ & & & 1\\ \hline 0 & & & \\ & -1 & & 	\end{array}\right],
\end{align*}
but here we can replace the bimap with the pseudo-isometric bimap represented by
$\{\Phi_1+\Phi_2,\Phi_2\}$ for which $\Phi_1+\Phi_2$ is nondegenerate.  When the field is sufficiently large, we can always
make such adjustments. In particular, the following holds.

\begin{prop}
Let $k$ be an infinite field, and $\circ\colon k^d\times k^d\bmto k^2$ a nondegenerate,
alternating $k$-bimap of genus 2. Then $\circ$ is sloped if, and only if, 
all of its indecomposable factors are sloped.
\end{prop}

\begin{proof}
The forward direction is clear.  
For the converse, suppose 
\begin{align*}
(i=1,2) & & \Phi_i=\diag(\Phi_i^{(1)},\ldots,\Phi_i^{(t)}),
\end{align*} 
represents $\circ$ and 
respects a fully-refined orthogonal decomposition, where
each $\{\Phi_1^{(j)},\Phi_2^{(j)}\}$ is sloped.
A linear combination of $\Phi_1,\Phi_2$ is nondegenerate if,
and only if, some evaluation of $\disc(\Phi_1,\Phi_2)\in k[x,y]$
does not vanish. 
By assumption, each $\disc(\Phi_1^{(i)},\Phi_2^{(i)})\neq 0$ (as a polynomial),
so $\disc(\Phi_1,\Phi_2)=\prod_i \disc(\Phi_1^{(i)},\Phi_2^{(i)})\neq 0$.
As $k$ is infinite,
there is a point not on the variety of $\disc(\Phi_1,\Phi_2)$.
\end{proof}

For finite fields, the situation is more delicate.

\begin{lemma}
For every finite field $\mathbb{F}_q$, there is an integer $d$ and an alternating bimap
$\circ\colon \mathbb{F}_q^d\times\mathbb{F}_q^d\bmto\mathbb{F}_q^2$,
all of whose indecomposable summands are sloped,
such that every pair $\{\Phi_1,\Phi_2\}$ of forms
representing $\circ$ consists of degenerate matrices.
\end{lemma}

\begin{proof}
Consider a pair $\{\Phi_1,\Phi_2\}$ representing an alternating bimap $\circ\colon
\mathbb{F}_q^d\times\mathbb{F}_q^d\bmto\mathbb{F}_q^2$. If no 
nondegenerate linear combination of $\Phi_1,\Phi_2$ exists, then clearly
the Pfaffian $\Pf(\Phi_2,\Phi_2)$ vanishes on all of $PG(1,\mathbb{F}_q)$.  
This means that $\prod_{\omega\in \mathbb{F}_q} (x-\omega y)\in\mathbb{F}_q[x,y]$ 
divides $\Pf(\Phi_1,\Phi_2)$.
For each $\omega\in \mathbb{F}_q$, 
\begin{align*}
	\Pf\left(\begin{bmatrix}
	0 & 1 \\ -1 & 0 
	\end{bmatrix},
	\begin{bmatrix}
	0 & \omega \\ -\omega & 0 
	\end{bmatrix}\right) & = x-\omega y.
\end{align*}
The orthogonal sum of all such pairs yields a pair of forms whose
discriminant vanishes on $PG(1,\mathbb{F}_q)$, but whose indecomposable summands
are all sloped.
\end{proof}

Fortunately, our analysis comes to the rescue. The following scholium 
allows us to treat $\mathbb{F}_q$ as a ``small" field whenever 
such a configuration occurs.

\begin{prop}
\label{coro:field-problem}
If $\circ\colon \mathbb{F}_q^d\times\mathbb{F}_q^d\bmto\mathbb{F}_q^2$ is an alternating non-sloped
bimap, all of whose indecomposable summands are sloped, then $q<d$.
\end{prop}

\subsection{Pfaffian isomorphism test for small fields}
\label{subsec:small}
Let $\circ,\bullet\colon \mathbb{F}_q^d\times\mathbb{F}_q^d\bmto
\mathbb{F}_q^2$ be two given bimaps whose indecomposable
summands are all sloped. Write each bimap relative to
a fully-refined orthogonal decomposition, as in Theorem~\ref{thm:det-method},
represented by a pair $\{\Phi_1,\Phi_2\}$ with
$\Phi_i=\diag\left(\Phi_i^{(1)},\dots,\Phi_i^{(s)}\right)$.  To each pair we associate a collection 
$\{\Pf(\Phi_1,\Phi_2)\colon 1\leq i\leq s\}$ of homogeneous polynomials.
Then, using Theorem~\ref{thm:det-method}, we can test whether
or not $\circ$ and $\bullet$ are pseudo-isometric by exhaustively checking
every element $\hat{\alpha}\in\GL(2,\mathbb{F}_q)$ to see
if it yields an equivalence between the two collections.
Moreover, Theorem~\ref{thm:det-method} is constructive: for suitable $\hat{\alpha}$
we can compute $\alpha\in\GL(d,\mathbb{F}_q)$ 
such that $(\alpha,\hat{\alpha})$ is a pseudo-isometry from $\circ$
to $\bullet$.

The complexity of the algorithm outlined above
contains an unavoidable factor of $|\PGL(2,\mathbb{F}_q)|$
for the exhaustive search. In practice it works well
when $q$ is small. We remark that, for a pair of bimaps
of the sort described in Proposition~\ref{coro:field-problem},
$q<d$, and we regard $d$ as small. Thus, we assume
in the remainder of this section
that alternating bimaps of genus 2 over large fields are sloped.

\subsection{Adjoint-tensor isomorphism test for large fields}
\label{subsec:large}
The shortcut isomorphism test described in the preceding section, while very effective in practical settings, has an unavoidable
factor of $O(q^3)$ in its complexity. Hence,
the performance of this technique deteriorates quickly as the size of $q$ increases, as the graphic in Figure~\ref{fig:ATvP}
clearly demonstrates. We therefore adapt the adjoint-tensor method to this more general setting.
\medskip

\begin{figure}[!htbp]
\input{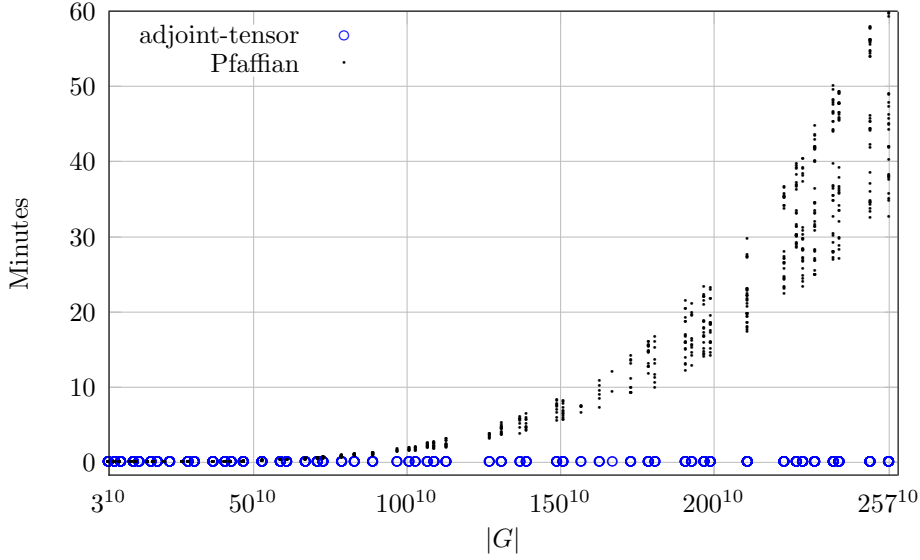}
\caption{Comparison of adjoint-tensor method with Pfaffian method for random $p$-groups of genus
2 as we let $p$ increase.}
\label{fig:ATvP}
\end{figure}
\medskip

The algorithm itself proceeds exactly as described in Section~\ref{subsec:sloped-indec}. 
Difficulties and complexity issues arise because the structure of 
$N^*(A)$ is more complex than in the indecomposable case, 
and because we handle that additional structure
by brute force. This gives rise to the rather less elegant complexity statement in 
Theorem~\ref{thm:bimap-alg}.
We now discuss all of the subtleties that arise in moving from indecomposable bimaps to 
general bimaps, and indicate
how we handle them.

The first subtlety occurs prior to the main algorithm, however. Recall, we have assumed that the 
indecomposable summands are all sloped. Our presentation of 
the adjoint-tensor method presumes that the given 
bimaps are sloped ``globally". This is the importance of Proposition~\ref{coro:field-problem}: 
if this happens not to be the case, then $q$ is small relative to $d$
and we use Section~\ref{subsec:small} instead.
\medskip

The crucial point suggested above is that, unlike the approach in 
Section~\ref{subsec:small}, we treat the input
bimaps globally (rather than working with indecomposable summands). The adjoint
algebras $A(\circ)$ and $A(\bullet)$ are still centralizers of single matrices,
and hence the conjugacy problem in Section~\ref{subsec:conj} goes through unchanged. 
\smallskip

Moving on to the tensor product $T=\Bbb{F}_q^d\wedge_A\Bbb{F}_q^d$, once again there is little
new to note. As observed in Remark~\ref{rem:general-slope}, Theorem~\ref{thm:tensor-bimap} applies to our new setting.
That is, if $\sigma=\Phi_2\Phi_1^{-1}$ is a slope of $\circ$, and $m(x)$ its minimal polynomial,
then $T\cong K=\Bbb{F}_q[x]/(m(x))$ as $\Bbb{F}_q$-modules, but now $K$ may have
multiple primary components and hence is not necessarily a local ring.
The proof of Theorem~\ref{thm:tensor-bimap}
is essentially the same except that we compute with each primary component separately.
\smallskip

We turn now to the structure of $\pseudo(\wedge_A)=N^*(A)$.
We refer, once again, to the general structure theorem in~\cite{BW:autotopism}*{Theorem 4.5}.

First, note that it's possible to have a subgroup of permutation matrices inside $N^*(A)$ arising from
the representation of $A$. More precisely, $N^*(A)$ permutes the isoptypic components
of the decomposition of $K$ into primary components. Recall that primary components
$V_i$ ($i=1,2$) are {\em isotypic} if they have minimal polynomial $p_i^n$, $p_i$
irreducible, with $\deg p_1=\deg p_2$ and where the components have identical Jordan block structures.
We denote this permutation subgroup of $N^*(A)$ by $\Pi$. It is possible, provided $q\geqslant \frac{d}{2}$,
for $|\Pi|$ to be as large as $\left(\frac{d}{2}\right)!$. 

Secondly, $K/J(K)$ is a product of fields
(as opposed to a single field). Therefore, the subgroup $\Gamma$ of
$\Gamma{\rm L}(1,K)$, which was previously a single Galois group, may
contain a direct product of Galois groups. Hence, $|\Gamma|$ may be
as large as $2^{\frac{d}{4}}$.
\smallskip

We now turn to the final step of the algorithm in Section~\ref{subsec:transporter}.
We proceed exactly as before, but instead of looping over $\Gamma$, we now
loop over $\Gamma\Pi$. Observe that both 
$|\Gamma|$ and $|\Pi|$ are bounded under the additional hypotheses 
of the last assertion in Theorem~\ref{thm:bimap-alg}, which is what yields polynomial time
in that case.

\subsection{The group of pseudo-isometries of a bimap of genus 2.}
\label{sec:pseudo-group}
Recall that our test for pseudo-isometry between given bimaps 
$\circ,\bullet\colon \mathbb{F}_q^d\times\mathbb{F}_q^d\bmto \mathbb{F}_q^2$
promises the set of {\em all} such pseudo-isometries (if such is needed).
It does so by additionally returning generators for the group
$\pseudo(\circ)$. In many applications of our methods, moreover,
this {\em is} the problem of interest. Thus we briefly describe how to
adapt the various tools for our pseudo-isometry test to solve
the following problem.
\bigskip

\begin{minipage}{0.9\textwidth}
\noindent {\sc PseudoIsometryGroup}$\;(~\circ~)$
\begin{description}
\item[Given] an alternating $\mathbb{F}_q$-bimap $\circ\colon\mathbb{F}_q^d\times\mathbb{F}_q^d\bmto \mathbb{F}_q^e$.
\item[Return]  (generators for) the group $\pseudo(\circ)$.
\end{description}
\end{minipage}
\bigskip

Again, we focus on genus 2, and
consider first the situation where $q$ is considered small, as in Section~\ref{subsec:small}.
Here, there is very little to be said. We proceed -- as though testing for pseudo-isometry between
$\circ$ and itself -- by listing all $\hat{\phi}\in\Gamma{\rm L}(2,\mathbb{F}_q)$ and testing
whether $\hat{\phi}$ lifts to a pseudo-isometry $(\phi,\hat{\phi})$ of $\circ$. When we have
exhausted the elements of $\Gamma{\rm L}(2,\mathbb{F}_q)$ we have the entire
group $\pseudo(\circ)/\isom(\circ)$.

Next, suppose that $q$ is large. Any pseudo-isometry of $\circ$ preserves a basic
decomposition of $\circ$ into its flat and sloped parts. 
We saw in Proposition~\ref{prop:pseudo-flat} that the pseudo-isometry group of
the flat part induces the full $\Gamma{\rm L}(2,\mathbb{F}_q)$ on $\mathbb{F}_q^2$,
and this result is constructive in that it provides a lift of any given $\hat{\phi}\in\Gamma{\rm L}(2,\mathbb{F}_q)$
to a pseudo-isometry of the flat part. Thus, in view of Section~\ref{sec:small-field}, 
it suffices to construct $\pseudo(\circ)$ when
$\circ\colon \mathbb{F}_q^d\times\mathbb{F}_q^d\bmto \mathbb{F}_q^2$ is sloped.

This, in fact, is somewhat easier than deciding pseudo-isometry because we need not concern ourselves 
with conjugating adjoint algebras. In fact, referring to the pseudo-code in Section~\ref{sec:adjten}, 
everything remains the same in an algorithm for {\sc PseudoIsometryGroup} until Line 5. 
Here, instead of seeking a single element $(\phi,\hat{\phi})\in\pseudo(\wedge_A)$ 
mapping $\ker\hat{\circ}$ to $\ker\hat{\star}$, we require the full stabilizer in $\pseudo(\wedge_A)$ of 
$\ker\hat{\circ}$, say. One solves such ``stabilizer'' problems using exactly the same machinery 
we used for the ``transporter" problems at no additional cost. 
\smallskip

In sum, we have proved the following.

\begin{thm}
\label{thm:pseudo-gp-alg}
There is a deterministic algorithm that, given an 
alternating bimap $\circ\colon \mathbb{F}_q^d\times\mathbb{F}_q^d\bmto
\mathbb{F}_q^2$ of genus 2, constructs generators for $\pseudo(\circ)$.
The algorithm is polynomial time if either $q$ is bounded, or
if the number of pairwise pseudo-isometric indecomposable summands
of the input bimaps is bounded.
\end{thm}



\section{Group Isomorphism and Automorphism}
\label{sec:iso-auto}
Finally, we are able to tie together our work thus
far and present isomorphism tests for groups.
We shall work within the framework of the following
two problems.

\bigskip
\begin{minipage}{0.9\textwidth} 
\noindent  {\sc Isoclinism}$_p\;(~G~,~H~)$
\begin{description}
\item[Given] finite $p$-groups $G$ and $H$ of $p$-class 2.
\item[Return] an isoclinism $G\to H$, if such exists.
\end{description}
\end{minipage}
\bigskip

By Proposition~\ref{prop:isoclinic-isomorphic} this implies solving the following special case.

\bigskip
\begin{minipage}{0.9\textwidth}
\noindent  {\sc Isomorphism}$_p\;(~G~,~H~)$ 
\begin{description} 
\item[Given] finite $p$-groups $G$ and $H$ of class 2 and exponent $p$.
\item[Return] an isomorphism $G\to H$, if such exists.
\end{description}
\end{minipage}
\bigskip

As with pseudo-isometries between bimaps, 
we can in principle produce {\em all} isoclinisms 
(or isomorphisms) between
$G$ and $H$ 
by returning the group
of all such for $G$. Hence, we
are also concerned with the following problems.

\bigskip
\begin{minipage}{0.9\textwidth} 
\noindent  {\sc AutoclinismGroup}$_p\;(~G~)$
\begin{description}
\item[Given] a finite $p$-group $G$ of $p$-class 2.
\item[Return] (generators for) the group of all autoclinisms of $G$.
\end{description}
\end{minipage}
\bigskip

\begin{minipage}{0.9\textwidth}
\noindent  {\sc AutomorphismGroup}$_p\;(~G~)$ 
\begin{description} 
\item[Given] a finite $p$-group $G$ of class 2 and exponent $p$.
\item[Return] (generators for) the group of all automorphisms of $G$.
\end{description}
\end{minipage}
\bigskip

Recall that groups are given in a succinct representation
by generating sets or specialized presentations.
Let $G$ and $H$ denote the given $p$-groups of class 2.
Using Proposition~\ref{prop:basic-algs} we find generators for $Z(G)$ and $G'$,
and construct systems of forms over $\mathbb{Z}_p$ representing
$\circ_G$ and $\circ_H$. Theorem~\ref{thm:baer} 
states that $G$ and $H$ are isoclinic if, and only if,
$\circ_G$ and $\circ_H$ are pseudo-isometric.
\smallskip

Next, compute the centroid, $C(\circ_G)$, 
of $\circ_G$ as the solution of a system of linear equations.
In the same way compute $C(\circ_H)$. Evidently, $\circ_G$
and $\circ_H$ are pseudo-isometric only if the centroid rings
are isomorphic and are represented identically on the Frattini 
quotients of their associated groups. Recall that the centroid
may be used to obtain a decomposition of a group into direct 
factors~\cite{Wilson:RemakI}, so we may assume that the groups
are directly indecomposable.
Using~\cite{BO}*{Section 2.2}, compute $J(C(\circ_G))$ and
write $C(\circ_G)/J(C(\circ_G))$ as a field extension, $\mathbb{F}_q$, of $\mathbb{Z}_p$.
Do the same for $C(\circ_H)$. Thus, we may assume that 
$C(\circ_H)/J(C(\circ_H))\cong \mathbb{F}_q$, and hence
that we have two bimaps 
$\circ_G,\circ_H\colon \mathbb{F}_q^d\times\mathbb{F}_q^d\bmto\mathbb{F}_q^e$.
\smallskip

To solve {\sc Isoclinism} for $G$ and $H$, we must solve
{\sc PseudoIsometry} for $\circ_G$ and $\circ_H$.
The results in 
Sections~\ref{sec:iso-ind-genus2} and~\ref{sec:iso-genus2} apply, of course,
to the case where $e=2$, so we translate them now into the 
language of groups. The following isoclinism test for centrally
indecomposable groups of genus 2 follows immediately
from Theorem~\ref{thm:indecomp-bimap-alg}.

\begin{thm}
\label{thm:indecomp-gp-alg}
There is a deterministic, polynomial-time algorithm that, given centrally indecomposable
$p$-groups $G$ and $H$ of genus 2, decides whether $G$ and $H$ are isoclinic and, if so, 
constructs an isoclinism $G\to H$.  
\end{thm}

The following extension to arbitrary groups of genus 2 follows from Theorem~\ref{thm:bimap-alg}.

\begin{thm}
\label{thm:gp-alg}
There is a deterministic algorithm that, given $p$-groups $G$ and $H$ of genus 2, 
decides whether $G$ and $H$ are isoclinic and, if so, constructs an isoclinism $G\to H$. 
The algorithm is polynomial time if either $|C(\circ_G)/J(C(\circ_G))|$ is bounded, or
if the number of pairwise isoclinic centrally indecomposable factors of $G$ and $H$ is bounded.
\end{thm}

We get our main theorem -- which, for convenience, we restate -- 
as a consequence of this result.

\begin{cor}[Theorem~\ref{thm:main1}]
There is a deterministic, polynomial-time algorithm that, given groups 
$G$ and $H$ as quotients of permutation groups, decides if $G$ is a $p$-group of 
exponent $p$-class $2$ with
commutator subgroup of order dividing $p^2$, and if so, decides if $G$ and $H$ are isoclinic.  
\end{cor}

\begin{proof}
A group of exponent $p$-class 2 and central commutator of order dividing $p^2$  
has genus at most $2$ over its centroid. 
The possible centroids are the $1$ and $2$-dimensional $\mathbb{Z}_p$-algebras: $\mathbb{Z}_p$, 
$\mathbb{Z}_p\oplus \mathbb{Z}_p\cong \mathbb{Z}_p[x]/(x(x-1))$, 
$\mathbb{F}_{p^2}\cong\mathbb{Z}_p[x]/(x^2+ax+b)$, or $\mathbb{Z}_p[x]/(x^2)$.

If  the centroid is $\mathbb{Z}_p$ then $G$ has genus $2$ and we apply Theorem~\ref{thm:gp-alg}.
In the other cases $G$ has genus $1$ over the centroid and so the associated bimap $\circ_G$ is represented
by a single nondegenerate form over the centroid.  For any $p$-group represented as
a permutation group, $p$ is bounded by the input length,
so the resulting algorithm is polynomial time in this setting.
In all other cases $\circ_G$ is an alternating nondegenerate form over a local ring with a
symplectic basis that is unique up to pseudo-isometry
(see, for example,~\cite{Milnor-Husemoller}*{Chapter I, Corollary 3.5}).
\end{proof}

In all of these results, if the input groups additionally have exponent $p$,
then ``isoclinism" may be replaced with ``isomorphism".
\medskip

We now turn to {\sc AutoclinismGroup} and {\sc AutomorphismGroup}. 
Let $G$ be a group of exponent $p$-class 2, and $\circ_G$ its associated
alternating bimap. We may again assume that $G$ is directly indecomposable,
and hence write $\circ_G\colon \mathbb{F}_q^d\times\mathbb{F}_q^d\bmto \mathbb{F}_q^e$,
where $C(\circ_G)/J(C(\circ_G))\cong\mathbb{F}_q$.
Then, by definition, the autoclinism group coincides with $\pseudo(\circ_G)$, so  
the results in Section~\ref{sec:pseudo-group} solve {\sc AutoclinismGroup}.  
In the case of a group of exponent $p$,
Proposition~\ref{prop:autoclinism}, gives an exact sequence
\begin{align*}
\xymatrix{
	1 \ar[r] &  C_{\Aut(G)}(Z(G)) \ar[r]^{\iota} & 	{\rm Aut}(G)\ar[r]^{\pi} & \pseudo(\circ_G).
}
\end{align*}
Generators for the $p$-group $C_{\Aut(G)}(Z(G))$ are defined by linear equations.
If $G$ is genus 2, then generators for $\pseudo(\circ_G)$ are
obtained in Section~\ref{sec:pseudo-group}. Hence, we have proved the following.

\begin{thm}
There is a deterministic algorithm that, given a
$p$-group $G$ 
of genus 2, 
constructs generators for the group of autoclinisms of $G$.
The algorithm is polynomial time if either $|C(\circ_G)/J(C(\circ_G))|$ is bounded, or
if the number of pairwise isoclinic centrally indecomposable factors
of $G$ is bounded. Again, ``autoclinism" may be replaced with ``automorphism" if $G$ has exponent $p$.
\end{thm}


\section{Implementation and Performance}
\label{sec:imp}

As mentioned in Section~\ref{sec:imp-intro}
we have implemented the algorithms
presented in Sections~\ref{sec:adjten}--\ref{sec:iso-auto} in the 
computer algebra system~{\sc magma}.
The code is available upon request from the authors.
Our implementation makes essential use of the ``StarAlgebra" 
package implemented by the first and third authors \cite{BW:isom}.
Although there are  
areas where performance can be improved, the plots in
Figures~\ref{fig:plot} and~\ref{fig:ATvP} 
illustrate the efficacy of our implementation.
All tests were carried out
on an Intel Xeon E5-1620, 3.60 GHz processor, running {\sc magma} V2.21--4.

We now comment briefly on aspects of the experiment whose
results are depicted in Figure~\ref{fig:plot} (henceforth referred to as Experiment A), 
and also on  further experiments designed to probe the behavior of the implementation 
when provided with input groups that, for one reason or another, are not likely to be isomorphic. 
\medskip

In Experiment A, we constructed random pairs of groups of genus 2 as follows. First, 
for fixed $d$, we generated 
a pair $\{\Phi_1,\Phi_2\}$ of skew-symmetric $d\times d$ matrices
with entries in $\Bbb{F}_5$. 
Next, 
we built a 5-group, $G$, of genus 2 as a PC-group with commutator relations determined by
the entries of $\{\Phi_1,\Phi_2\}$.
We then chose a random element
$g\in\GL(d,\Bbb{F}_5)$ and $h\in\GL(2,\mathbb{F}_5)$, computed $\{\Psi_1,\Psi_2\}=\{g\Phi_1g^{{\rm tr}},g\Phi_2g^{{\rm tr}}\}^h$, and used these
matrices to define another PC-group, $H$ which, by construction, is isomorphic to $G$. 
Finally, we used our implementation to test for isomorphism between $G$ and $H$.
The test was repeated 10 times for each even $d$ between 4 and 254, and twice for each odd
$d$ between 3 and 135, for a total 1260 tests. Necessarily all tests for $d$ odd produced
flat groups; by construction all of the groups for even $d$ were sloped. 
\smallskip

We have shown that the asymptotic complexity of our algorithms is $O(d^{2\omega})$,
the complexity of solving systems of linear equations in $(1+o(1))d^2$ equations and variables, we call this ``$d^2$ linear algebra''. 
We mentioned in Section 1
that Figure~\ref{fig:plot} suggests our runtimes agree with theory in the case of sloped input,
and claimed that the same is true for flat groups. Considering each input type (sloped and flat) 
separately we computed the logarithm of the completion times for each test, and formed the ratio
with the logarithm of the completion time for a random $d^2$ linear algebra.
The results recorded in Figure~\ref{fig:logplot} show that indeed both trend towards 1
and so track with $d^2$ linear algebra.  The flat case trends much more slowly
to $1$ than does the sloped case leaving room for future improvement.
\medskip

\begin{figure}[!htbp]
\input{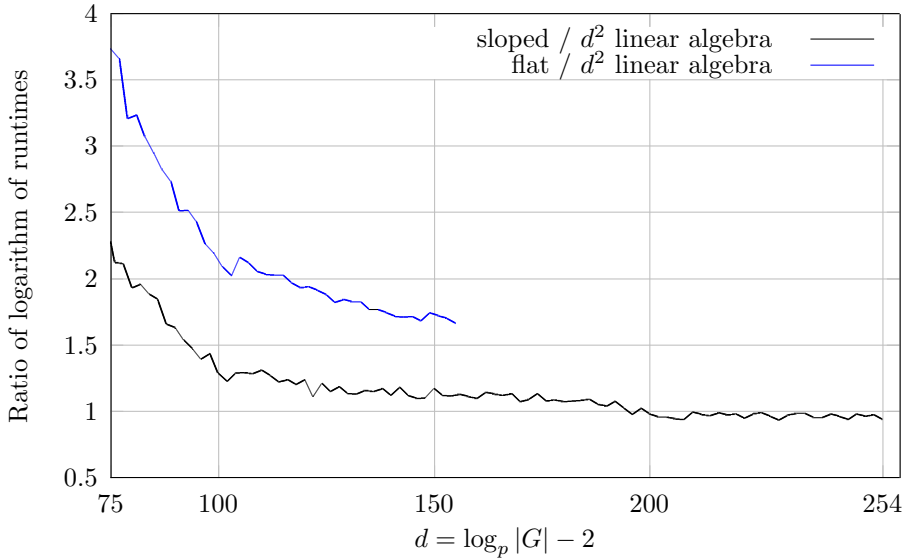}
\caption{Graph showing the ratio of the logarithm of our runtimes against the logarithm of
the time needed to solve $d^2\times d^2$ linear systems as $d=\log_p|G|-2$ increases.}
\label{fig:logplot}
\end{figure}

We comment briefly on the rather large variance in runtimes 
one clearly sees from the spikes in Figure~\ref{fig:plot} for the sloped
case (in contrast to the smooth graph for the flat groups).
An essential component for both types of groups is the construction
and manipulation of the adjoint algebra. In the case of a sloped group, $G$,
one constructs $A(\circ_G)$ very quickly using the methods of~\cite{BW:slope}.
For these groups, the completion time is affected if the Jacobson
radical of $A(\circ_G)$ is nontrivial, or if the natural module for $A(\circ_G)$
decomposes into many blocks. Information about the number of 
blocks and sizes of the largest blocks for the groups in our experiment are given
in Figures~\ref{fig:blockstructure} and~\ref{fig:blockstructure2}, respectively.
As one can see, for sloped groups there is considerable variability in the block structure.
With flat groups, on the other hand, there is usually just
a single indecomposable block, and the runtime is always dominated by the 
construction of the adjoint algebra.
This accounts for smooth graph for flat groups and spiky graph for sloped groups.
\medskip

\begin{figure}[!htbp]
\includegraphics{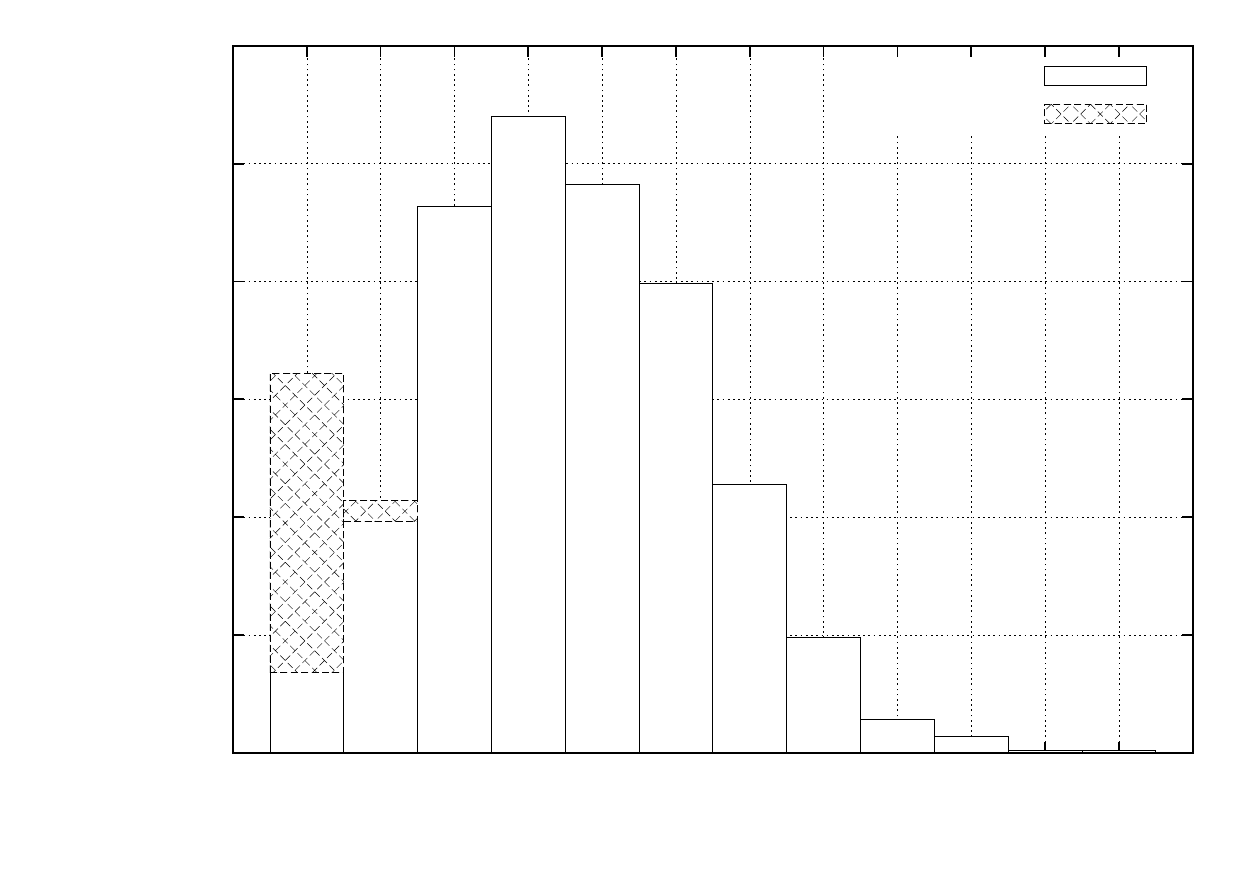}
\caption{Bar graph showing the number of indecomposable summands
of $A(\circ_G)$ for each group $G$ in Experiment A.}
\label{fig:blockstructure}
\end{figure}

\begin{figure}[!htbp]
\input{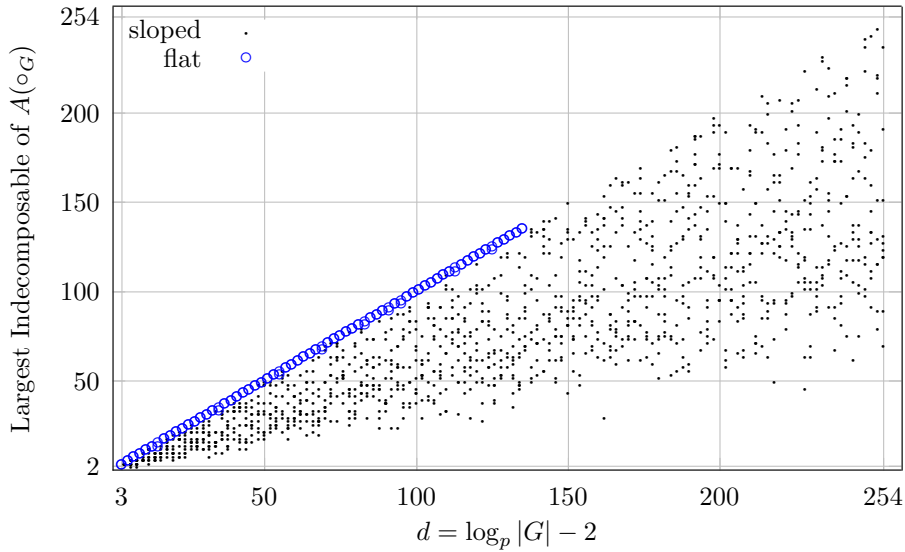}
\caption{Graph showing for each group $G$ in Experiment A
the dimension of the largest indecomposable summand for $A(\circ_G)$.}
\label{fig:blockstructure2}
\end{figure}

We conclude with some brief remarks concerning the
behavior of our implementation when given pairs of input groups that are unlikely
to be isomorphic. To explore this matter we conducted two additional experiments.
\smallskip

In Experiment B, we fixed $p=1021$ and even $d$ between $20$ and $40$,
and proceeded as in Experiment A to build a
$p$-group, $G$, of genus 2 with $d=\log_p|G|-2$ from a random pair of alternating forms.
We constructed a group $H$ from an 
independent pair of alternating forms. Unsurprisingly, none of the pairs of groups
in our experiment were isomorphic and our implementation quickly determined
this by showing that the adjoint algebras $A(\circ_G)$ and $A(\circ_H)$ 
were non-conjugate.

Experiment C was designed to produce non-isomorphic pairs
of groups having conjugate adjoint algebras. It proceeded as follows.
We built a random sloped $p$-group, $G$, of genus 2 as above using a pair
of alternating forms. (This time we used $p=3$ and let $d$ vary between 8 and 20.)
Next, we computed $A=A(\circ_G)$, formed the tensor product
$\Bbb{Z}_p^d\otimes_A\Bbb{Z}_p^d$, and produced another sloped pair
of forms by taking a random 2-dimensional projection of this tensor product.
By construction the group built from such a pair of forms has adjoint algebra $A$. 
Of the 21 pairs of groups
we built in this way, there were 6 isomorphic pairs. Of the remaining 15
non-isomorphic pairs, 9 were indistinguishable by any of the usual isomorphism invariants.
As the results in~\cite{LW} clearly demonstrate, though, we should not be surprised
by the failure of standard techniques to decide non-isomorphism in such a context.


\section{Closing remarks}

We conclude with several observations about the results in this paper,
their connection to ongoing projects, and implications for possible future research.  
\bigskip

\noindent {\bf Counting.}~ We have expended considerable effort developing
theory for the groups of genus 2, and algorithms to decide isomorphism in
this class of groups. It seems natural to wonder, then, how many such groups there are.
The following result shows that the class is not polynomially bounded.

\begin{prop}
\label{prop:count}
The number of pairwise non-isoclinic groups of order $p^n$ that have genus 
$2$ over a field is $p^{n/2+\Theta(1)}$.
\end{prop}

\begin{proof}
We have seen that a group of genus $2$ over a field $\mathbb{F}_q$ is determined, up
to isoclinism, by a pair $\{\Phi_1,\Phi_2\}$ of alternating forms of dimension $m$
over $\Bbb{F}_q$. For a group 
of order $p^n$ and genus $2$ 
it follows that
$n=(2m+2)\log_p q$. Furthermore, these can be written uniquely as an orthogonal 
sum sloped and flat components.  Let $m=s+f$ where $s$ is the dimension of 
the sloped factor.
 
To estimate the number of possibilities for flat portion we need only consider the dimensions of the flat 
indecomposable constituents, forming a partition of the total dimension, $f$.  
Furthermore, all flat indecomposables pairs of forms have odd dimension greater 
than $1$. so the number of flat indecomposables is at most the 
number of decompositions $f=\sum_i 2m_i+1$.  This is bounded by the number of 
partitions of $f/2$, and so is not more than $2^{f/2}$.

We now estimate the possibilities for the sloped portion.  We may assume
\[
\Psi_1=\begin{bmatrix} 0 & I_{s/2} \\ -I_{s/2} & 0 \end{bmatrix}~~\mbox{and}~~
\Psi_2=\begin{bmatrix} 0 & J \\ -J^{{\rm tr}} & 0 \end{bmatrix},
\]
where $J$ in generalized Jordan normal form.  
By a classical result of Frobenius and Hall, the number of Jordan forms  
(also the number of conjugacy classes in $\M_{s/2}(q)$) is $q^{s/2+o(1)}$.  
By Theorem~\ref{thm:det-method}
two sloped pairs determine isomorphic groups if, and only if, they 
are equivalent under the action of $\Gamma{\rm L}(2,\Bbb{F}_q)$. Hence, the total 
number of pairwise nonisomorphic sloped components is between $q^{s/2-4}$ and $q^{s/2}$.

The total number of pairwise nonisomorphic groups of genus $2$ and order $p^n$ is 
 maximized when $\Bbb{F}_q=\mathbb{Z}_p$
and $f\in O(1)$, resulting in the stated estimate.  
\end{proof}

We turn briefly to the degrees of permutation representations having groups of genus 
$2$ as a quotient. It transpires that if the centrally indecomposable terms have small 
degree representations as quotients then so does the entire group.

\begin{prop}
\label{prop:degree}
Let $G$ be a $p$-group of genus $2$ over a field $\mathbb{F}_{q}$ with fully refined central decomposition 
$\{G_1,\cdots,G_{\ell}\}$. Then $G$ has a faithful representation as a quotient of a permutation group in degree 
\begin{align*}
	\deg(G) & \leq \sum_{i=1}^{\ell} \deg(G_i), & \deg(G_i) & \leq \left\{\begin{array}{cc} 
	q^{2c_i \deg a_i(x)}, & H(\mathbb{F}_{q}[x]/(a_i(x)^{c_i}) )\twoheadrightarrow G_i;\\
	q^{2m+2}, & G_i\cong H_m^{\flat}(q).
	\end{array}\right.
\end{align*}
Furthermore, $|G|=q^{2s}\prod_{i=1}^{\ell} \deg (G_i)$, where $s$ the number of $G_i$ that are sloped.
\end{prop}
\begin{proof}
If $H=H(K)$, $K$ a commutative ring, then the stabilizer of $(1,s,t)$ in $H$ is 
\[
\left\{\begin{bmatrix} 1 & 0 & fs \\ 0 & 1 & f\\ 0 & 0& 1\end{bmatrix} : f\in K\right\}.
\]
Hence, the stabilizers of $(1,0,0)$ and $(1,1,0)$ intersect trivially, so the permutation 
representation of $H$ on $\{(1,s,t): s,t\in K\}$ is
faithful and transitive of degree $2|K|$.  It follows, for each $1\leq i\leq\ell$, that
\[
\deg(H(\mathbb{F}_q[x]/(a_i(x)^{c_i})))\leq q^{2c_i \deg a_i(x)}.
\]
A similar estimate holds for flat indecomposables, but here the representation is regular.     
Since central products are quotients of direct product the claim follows.
\end{proof}
\bigskip

\noindent {\bf Groups of genus 2 over local rings.}~
We introduced the concept of genus over an
arbitrary Artinian commutative ring. Via Theorem~\ref{thm:centroid-direct} we obtained
an immediate reduction to Artinian local commutative rings. At that stage,
due largely to the considerable complications that already exist, we focused
on fields.  We note, however, that an alternative approach exists for arbitrary local rings.
If the centroid of a given group $G$ has a nontrivial radical we can use this to 
build a proper nontrivial characteristic subgroup of $G$.  This is done, for example, by the last author 
in \cite{Wilson:alpha}, thereby reducing to a smaller instance of isomorphism. Thus, there is no 
need to extend our methods to the local case.  Algorithms for just this task were developed 
by the second author for the computer algebra system {\sc Magma} in~\cite{Maglione:filter}, 
and their performance is on par with the performance of our algorithms here.
\bigskip

\noindent {\bf Groups of higher genus.}~ As suggested in Section~\ref{sec:adjten},
the adjoint-tensor method is designed to work in much greater 
generality than genus $2$. In~\cite{LW}, for example, a version of the algorithm handles isomorphism of all 
quotients of Heisenberg groups over fields in time $O((\log |G|)^6)$. What prevents 
us from saying more is that one cannot predict the complexity of group isomorphism 
for a class of groups by the
adjoint-tensor method without a priori knowledge of the associated adjoint rings.  In the case 
of quotients of Heisenberg groups $H_m(K)$, $K$ a local Artinian ring,  the adjoint rings
 are generically the same as the adjoint ring of $H_m(K)$, which is none other than 
$M_{2m}(K)$.  (This follows from  a Galois correspondence explained in \cite{BW:autotopism}).  
So long as this ring is manageable -- say, if the centroid has a decideable isomorphism problem and a 
bounded number of common weight spaces -- then some variation of our analysis still applies.  

Our point is that adjoint-tensor is a general technique that works exceedingly well when circumstances favor it,
and the computational effort required to decide when such circumstances exist is relatively inconsequential. 
Thus, it does no harm to attempt isomorphism by this method before resorting to the general 
techniques.  Our implementation includes some rudimentary estimation functions to predict the likelihood 
that our method will improve performance.  One might eventually analyze the complexity of adjoint-tensor
in the general case, but we caution it will become difficult if the centroid is an arbitrary local ring.

\section*{Acknowledgments}

The authors wish to thank R. Lipyanski and V. Serge{\u\i}{\v{c}}uk for their most helpful guidance
on wild and tame classification problems.

This work was partially supported by a grant from the Simons Foundation (\#281435 to Peter Brooksbank).


\end{document}